\documentclass{amsart}
\usepackage[margin=1in]{geometry}
\usepackage[all]{xy}
\usepackage{verbatim}
\usepackage{color}
\usepackage{amsthm}
\usepackage{amssymb}
\usepackage[colorlinks=true]{hyperref}
\usepackage{footmisc}
\usepackage{mathtools}

\numberwithin{equation}{section}

\newtheorem{theorem}[equation]{Theorem}
\newtheorem*{theorem*}{Theorem} 
\newtheorem{lemma}[equation]{Lemma}

\newtheorem*{conjecture*}{Mamma Conjecture}
\newtheorem*{conjecture1*}{Mamma Conjecture (revisited)}
\newtheorem{proposition}[equation]{Proposition}
\newtheorem{corollary}[equation]{Corollary}
\newtheorem*{corollary*}{Corollary}

\theoremstyle{remark}

\newtheorem{example}[equation]{Example}

\newtheorem{notation}[equation]{Notation}

\theoremstyle{remark}
\newtheorem{remark}[equation]{Remark}

\newcommand{\cA}{{\mathcal A}}

\newcommand{\cC}{{\mathcal C}}

\newcommand{\cN}{{\mathcal N}}

\newcommand{\cP}{{\mathcal P}}

\newcommand{\cR}{{\mathcal R}}

\newcommand{\cZ}{{\mathcal Z}}

\newcommand{\bbA}{\mathbb{A}}

\newcommand{\bbF}{\mathbb{F}}

\newcommand{\bbH}{\mathbb{H}}
\newcommand{\bbL}{\mathbb{L}}
\newcommand{\bbN}{\mathbb{N}}
\newcommand{\bbP}{\mathbb{P}}
\newcommand{\bbR}{\mathbb{R}}

\newcommand{\bbQ}{\mathbb{Q}}
\newcommand{\bbZ}{\mathbb{Z}}

\DeclareMathOperator{\id}{id}

\DeclareMathOperator{\NChow}{NChow} 
\DeclareMathOperator{\NNum}{NNum} 

\DeclareMathOperator{\Num}{Num} 

\newcommand{\dgcat}{\mathrm{dgcat}} 

\newcommand{\perf}{\mathrm{perf}}

\newcommand{\Chow}{\mathrm{Chow}}

\newcommand{\dg}{\mathrm{dg}}

\newcommand{\Hom}{\mathrm{Hom}}

\newcommand{\op}{\mathrm{op}}

\newcommand{\too}{\longrightarrow}

\let\oldmarginpar\marginpar
\def\marginpar#1{\oldmarginpar{\tiny #1}}

\begin{document}

\title[Grothendieck classes of quadrics and involution varieties]{Grothendieck classes of\\ quadrics and involution varieties}
\author{Gon{\c c}alo~Tabuada}
\address{Gon{\c c}alo Tabuada, Mathematics Institute, Zeeman Building, University of Warwick, Coventry CV4 7AL UK.}
\email{goncalo.tabuada@warwick.ac.uk}
\urladdr{https://homepages.warwick.ac.uk/~u1972846/}
\date{\today}

\dedicatory{To Yuri I. Manin on the occasion of his $85^{\mathrm{th}}$ birthday, with my deepest admiration and gratitude.}

\abstract{In this article, by combining the recent theory of noncommutative motives with the classical theory of motives, we prove that if two quadrics (or, more generally, two involution varieties) have the same Grothendieck class, then they have the same even Clifford algebra and the same signature. As an application, we show in numerous cases (e.g., when the base field is a local or global field) that two quadrics (or, more generally, two involution varieties) have the same Grothendieck class if and only if they are isomorphic.}
}

\maketitle


\vspace{-0.3cm}

\noindent
{\it Adoramos a perfei{\c c}\~ao, porque n\~ao a podemos ter; repugna-la-\'iamos se a tiv\'essemos.\\ O perfeito \'e o desumano porque o humano \'e imperfeito.}

\vspace{0.1cm}

\noindent
Bernardo Soares, Livro do Desassossego.

\vspace{0.6cm}

\section{Introduction}\label{sec:intro}
Let $k$ be a field and $\mathrm{Var}(k)$ the category of {\em varieties}, i.e., reduced separated $k$-schemes of finite type. The {\em Grothendieck ring of varieties $K_0\mathrm{Var}(k)$}, introduced in a letter from Grothendieck to Serre (consult \cite[letter of 16/08/1964]{SG}), is defined as the quotient of the free abelian group on the set of isomorphism classes of varieties $[X]$ by the ``scissor'' relations $[X]=[Z]+[X\backslash Z]$, where $Z$ is a closed subvariety of $X$. The multiplication law is induced by the product of varieties. Despite the efforts of several mathematicians (consult, for example, the articles \cite{Bittner,Borisov, Lunts, Manin, Sebag, Naumann,Poonen} and the references therein), the structure of the Grothendieck ring of varieties still remains nowadays poorly understood. In this article, in order to better understand the structure of the Grothendieck ring of varieties, we address the following question:

\smallskip

{\em Question: For which pairs of varieties $X$ and $Y$ does the implication $[X]=[Y] \Rightarrow X\simeq Y$ holds?} 

\smallskip

\noindent
Such implication does not holds in general because there are numerous identifications that occur in the Grothendieck ring of varieties. For example, when $X$ and $Y$ are piecewise isomorphic\footnote{A concrete example is given by the ordinary cusp $X:=\mathrm{Spec}(k[x,y]/\langle y^2 - x^3\rangle)$ and the affine line $Y:=\bbA^1$. Although $X$ and $Y$ are {\em not} isomorphic, we have $[X]=[Y]$ because they are piecewise isomorphic.}, we have $[X]=[Y]$. Moreover, as explained in \cite[Chap.~2 \S6.3]{Integration}, there also exist varieties which are not piecewise isomorphic but which still have the same Grothendieck class\footnote{A concrete example, in characteristic zero, is the following: let $V$ be a $k$-vector space of dimension $7$ and $W \subset \bigwedge^2 V^\vee$ a generic subspace of dimension $7$. Note that, by definition, an element $w$ of $W$ corresponds to a skew-symmetric bilinear form $\varphi_w$ on $V$. Under these identifications, we can consider the varieties $X:=\{P \subset V\,|\, {\varphi_w}_{|P}=0\,\,\,\forall w \in W\} \subset \mathrm{Gr}(2,V)$ and $Y:=\{w \in W\,|\,\mathrm{rank}(\varphi_w)< 6\} \subset \bbP(W)$, where $\mathrm{Gr}(2,V)$ stands for the Grassmannian of $2$-dimensional subspaces in $V$. As proved by Borisov in \cite{Borisov}, we have $[X\times \bbA^6]=[Y \times \bbA^6]$ although $X \times \bbA^6$ and $Y \times \bbA^6$ are {\em not} piecewise isomorphic.}. In this article we show that, surprisingly, the aforementioned implication holds nevertheless for numerous pairs of quadrics and involution varieties!
\subsection*{Statement of results - quadrics}
Let $k$ be a field of characteristic zero. Given a (finite-dimensional) non-degenerate quadratic form $q\colon V \to k$, let us write $\mathrm{dim}(q)$ for its dimension, $\delta(q)\in k^\times/(k^\times)^2$ for its discriminant (when $\mathrm{dim}(q)$ is even), $C_0(q)$ for its even Clifford algebra, and $Q_q\subset \bbP(V)$ for the associated quadric; consult \cite[\S V]{Lam}. Recall from \cite[\S V Thm.~2.4]{Lam} that when $\mathrm{dim}(q)$ is odd, $C_0(q)$ is a central simple $k$-algebra; that when $\mathrm{dim}(q)$ is even and $\delta(q)\notin (k^\times)^2$, $C_0(q)$ is a central simple algebra over its center $k(\sqrt{\delta(q)})$; and that when $\mathrm{dim}(q)$ is even and $\delta(q)\in (k^\times)^2$, $C_0(q)\simeq C_0(q)^+\times C_0(q)^-$ is a product of two isomorphic central simple $k$-algebras. Recall also that $\mathrm{dim}_k(C_0(q))=2^{\mathrm{dim}(q)-1}$ and $\mathrm{dim}(Q_q)=\mathrm{dim}(q)-2$. Finally, when $k$ is formally-real, we will write $\mathrm{sgn}_P(q)\in \bbZ$ for the signature of $q$ with respect to an ordering $P$ of $k$; consult \cite[\S VIII]{Lam}.
\begin{theorem}\label{thm:main}
Let $q$ and $q'$ be two non-degenerate quadratic forms. If $[Q_q]=[Q_{q'}]$, then the following holds:
\begin{itemize}
\item[(i)] We have $\mathrm{dim}(q)=\mathrm{dim}(q')$.
\item[(ii)] We have $(\delta(q) \in (k^\times)^2) \Leftrightarrow (\delta(q')\in (k^\times)^2)$.
\item[(iii)] We have $C_0(q)\simeq C_0(q')$.
\item[(iv)] When $k$ is formally-real, we have $|\mathrm{sgn}_P(q)|=|\mathrm{sgn}_P(q')|$ for every ordering $P$ of $k$.
\end{itemize}
\begin{remark}\label{rk:discriminant}
Note that if $C_0(q)\simeq C_0(q')$, then $\delta(q)=\delta(q')$.
\end{remark}
\end{theorem}
Intuitively speaking, Theorem \ref{thm:main} shows that the dimension, the discriminant, the even Clifford algebra, and the absolute value of the signature, of a quadratic form are preserved by the ``scissor'' relations. Among other ingredients, the proof of items (ii)-(iii), resp. item (iv), makes use of the recent theory of noncommutative motives, resp. of the classical theory of motives\footnote{The first article on the theory of motives was written by Yuri I. Manin \cite{Manin1} in the sixties.}; consult \S\ref{sub:motives}-\S\ref{sub:radical} below. 

\smallskip

Theorem \ref{thm:main} enables the following applications:
\begin{theorem}\label{thm:applications}
Let $q$ and $q'$ be two non-degenerate quadratic forms. If $[Q_q]=[Q_{q'}]$, then the following holds:
\begin{itemize}
\item[(i)] When $\mathrm{dim}(q)\leq4$ (or, equivalently, $\mathrm{dim}(q')\leq4$), we have $Q_q\simeq Q_{q'}$.
\item[(ii)] When $I(k)^3$ is torsion-free, where $I(k)$ stands for the fundamental ideal of the Witt ring of quadratic forms $W(k)$, we have $Q_q\simeq Q_{q'}$. In the case where $k$ is formally-real and $\mathrm{dim}(q)$ is even (or, equivalently, $\mathrm{dim}(q')$ is even), we assume moreover that the Hasse-number $\tilde{u}(k)$ of $k$ is finite.
\end{itemize}
\end{theorem}
\begin{remark}
\begin{itemize}
\item[(a)] When $k$ is not formally-real, the Witt ring $W(k)$ is torsion; consult \cite[\S5]{Book-quadrics}. Consequently, $I(k)^3$ is torsion-free if and only if $I(k)^3=0$.
\item[(b)] All the assumptions of Theorem \ref{thm:applications} hold when $k$ is a local or global field, or a field of transcendence degree $\leq 1$ over a real-closed field or over an algebraically closed field; consult \cite[\S V-\S VI]{Book-quadrics}.
\end{itemize}
\end{remark}
Item (i) of Theorem \ref{thm:applications} shows that, in dimensions $\leq 2$, two quadrics have the same Grothendieck class if and only if they are isomorphic! In other words, we have the following inclusion:
\begin{equation}\label{eq:inclusion}
\frac{\{\text{Quadrics}\,\,Q_q\,|\,\mathrm{dim}(Q_q)\leq2\}}{\text{isomorphism}} \subset K_0\mathrm{Var}(k) \,.
\end{equation}
\begin{remark}[Quaternion algebras]\label{rk:quaternions}
The assignment $(a,b) \mapsto C(a,b):=(-ax^2 -by^2 + abu^2 =0)\subset \bbP^2$ induces a bijection between the set of quaternion $k$-algebras up to isomorphism and the set of quadrics in $\bbP^2$ up to isomorphism (a.k.a. conics). In the same vein, the assignment $(a,b) \mapsto (x^2 - ay^2 - bu^2 +(ab)w^2=0)\subset \bbP^3$ induces a bijection between the set of quaternion $k$-algebras up to isomorphism and the set of quadrics in $\bbP^3$ with trivial discriminant up to isomorphism. 
\end{remark}
\begin{example}[Quadrics of dimension $\leq 2$ over $\bbR$]
When $k=\bbR$, there are two quadrics in $\bbP^2$ up to isomorphism, namely $C(1,1)\simeq \bbP^1$ and $C({\bf H})$, where ${\bf H}:=(-1,-1)$ stands for Hamilton's $\bbR$-algebra of quaternions. In $\bbP^3$ there are three quadrics up to isomorphism, namely $(x^2-y^2 - u^2 + w^2=0)\subset \bbP^3$, $(x^2+y^2 + u^2 + w^2=0)\subset \bbP^3$ and $(x^2+y^2 + u^2 - w^2=0)\subset \bbP^3$. Making use of the above inclusion \eqref{eq:inclusion}, we hence conclude that the Grothendieck classes of these quadrics remain distinct in $K_0\mathrm{Var}(\bbR)$. Note that since $\bbR^\times/(\bbR^\times)^2\simeq \{\pm 1\}$, the discriminant of the quadric $(x^2+y^2 + u^2 - w^2=0)\subset \bbP^3$ is non-trivial.
\end{example}
\begin{example}[Quadrics of dimension $\leq 2$ over $\bbQ_p$]\label{ex:adic}
When $k=\bbQ_p$, with $p\neq 2$, there are two quadrics in $\bbP^2$ up to isomorphism, namely $C(1,1)\simeq \bbP^1$ and $C(\epsilon, p)$, where $\epsilon$ is a(ny) unit of $\bbQ_p^\times$ such that $\overline{\epsilon}$ is not a square in $(\bbZ/p\bbZ)^\times$. In $\bbP^3$ there are six quadrics up to isomorphism:
\begin{eqnarray}
(x^2 -  y^2 - u^2 + w^2=0)\subset \bbP^3  && (x^2 - \epsilon y^2 - pu^2 + (\epsilon p) w^2=0)\subset \bbP^3 \label{eq:quadric1}\\
(x^2 - y^2 - u^2 + \epsilon w^2=0)\subset \bbP^3 && (x^2 - y^2 - u^2 + p w^2=0)\subset \bbP^3 \label{eq:quadric2}\\
(x^2 - y^2 - u^2 + (\epsilon p) w^2=0)\subset \bbP^3 && (x^2 + y^2 - \epsilon u^2 - p w^2=0)\subset \bbP^3 \,.\label{eq:quadric3}
\end{eqnarray}
Making use of the above inclusion \eqref{eq:inclusion}, we hence conclude that the Grothendieck classes of these quadrics remain distinct in $K_0\mathrm{Var}(\bbQ_p)$. Note that since $\bbQ_p^\times/(\bbQ_p^\times)^2 = \{1, \epsilon, p, \epsilon p\}$, the discriminant of the quadrics \eqref{eq:quadric2}-\eqref{eq:quadric3} is non-trivial.
\end{example}
\begin{example}[Quadrics of dimension $\leq 2$ over $\bbQ$]
When $k=\bbQ$, there are infinitely many quadrics in $\bbP^2$ up to isomorphism. More specifically, following Remark \ref{rk:quaternions}, there is a bijection between the set of quaternion $\bbQ$-algebras up to isomorphism and the set of those positive integers which are not squares; consult \cite[\S III-IV]{Serre}. Under such bijection, the prime numbers $p$ which are congruent to $3$ modulo $4$ correspond to the quaternions algebras $(-1,-p)$. Consequently, we have, for example, the {\em infinite} family of {\em non-isomorphic} quadrics $\{C(-1,-p)\}_{p \equiv 3 (\mathrm{mod}\,4)}$. In the same vein, we have, for example, the {\em infinite} family of {\em non-isomorphic} quadrics with trivial discriminant $\{(x^2 + y^2 + p u^2 + p w^2 =0)\subset \bbP^3\}_{p \equiv 3 (\mathrm{mod}\,4)}$. Note that we have also, for example, the {\em infinite} family of {\em non-isomorphic} quadrics with non-trivial discriminant $\{(x^2+ y^2 + u^2 + p u^2=0)\subset \bbP^3\}_p$ parametrized by all prime numbers $p$. Making use of the above inclusion \eqref{eq:inclusion}, we hence conclude that the Grothendieck classes of all the aforementioned quadrics remain distinct in $K_0\mathrm{Var}(\bbQ)$.
\end{example}
Item (ii) of Theorem \ref{thm:applications} shows that when $I(k)^3$ is torsion-free (and $\tilde{u}(k)<\infty$), two quadrics have the same Grothendieck class if and only if they are isomorphic! Consequently, the above inclusion \eqref{eq:inclusion} admits the following extension to all dimensions:
\begin{eqnarray}\label{eq:inclusion2}
\frac{\{\text{Quadrics}\,\,Q_q\}}{\text{isomorphism}} \subset K_0\mathrm{Var}(k) && \mathrm{with}\,\,I(k)^3\,\,\mathrm{torsion}\text{-}\mathrm{free}\,\,\,(\mathrm{and}\,\,\tilde{u}(k)<\infty)\,.
\end{eqnarray}
\begin{example}[Quadrics over $\bbR$]
When $k=\bbR$, there are $\lfloor\frac{n+3}{2}\rfloor$ quadrics in $\bbP^n$ up to isomorphism, namely the following family $\{(x^2_1 + \cdots + x^2_i - x^2_{i+1} - \cdots - x^2_{n+1}=0) \subset \bbP^n\}_i$, with $\frac{n+1}{2} \leq i \leq n+1$ when $n$ is odd, and with $\lfloor\frac{n+3}{2} \rfloor \leq i \leq n+1$ when $n$ is even. Making use of the above inclusion \eqref{eq:inclusion2}, we hence conclude that the Grothendieck classes of these quadrics remain distinct in $K_0\mathrm{Var}(\bbR)$.
\end{example}
\begin{example}[Quadrics over $\bbQ_p$]
When $k=\bbQ_p$, with $p\neq 2$, there are finitely many quadrics in $\bbP^n$ up to isomorphism. More specifically, since the $u$-invariant $u(\bbQ_p)$ of $\bbQ_p$ is equal to $4$ (consult \cite[\S VI]{Book-quadrics}), there are two, resp. six, quadrics in $\bbP^n$ up to isomorphism when $n$ is even, resp. odd. The corresponding quadratic forms (well-defined up to similarity) are obtained by taking the orthogonal sum of the corresponding quadratic forms of Example \ref{ex:4} with a finite number of copies of the hyperbolic plane. Making use of the inclusion \eqref{eq:inclusion2}, we hence conclude that the Grothendieck classes of these quadrics remain distinct in $K_0\mathrm{Var}(\bbQ_p)$.
\end{example}
\begin{example}[Quadrics over $\bbQ$]
When $k=\bbQ$, there are infinitely many quadrics in $\bbP^n$ up to isomorphism. Note that the assignment $m \mapsto (\mathrm{sign}(m), \{r_p\}_{p})$, where $m=\pm \prod_{p\,\mathrm{prime}}p^{r_p}$ is the prime factorization, gives rise to a group isomorphism $\bbQ^\times/(\bbQ^\times)^2 \simeq (\{\pm 1\}\times \oplus_{p} \bbZ)/(\{1\} \times \oplus_{p} 2\bbZ)$. Consequently, when $n$ is odd, we have, for example, the following {\em infinite} family of {\em non-isomorphic} quadrics $\{(x^2_1 + \cdots + x^2_n + \lambda x^2_{n+1}=0) \subset \bbP^n\}_\lambda$ parametrized by the square classes $\lambda \in \bbQ^\times/(\bbQ^\times)^2$. Making use of the above inclusion \eqref{eq:inclusion2}, we hence conclude that the Grothendieck classes of all these quadrics remain distinct in $K_0\mathrm{Var}(\bbQ)$.
\end{example}
\subsection*{Statement of results - involution varieties}
Let $k$ be a field of characteristic zero. Given a central simple $k$-algebra $A$, let us write $\mathrm{deg}(A)$ for its degree and $\mathrm{SB}(A)$ for the associated Severi-Brauer variety. In the same vein, given a central simple $k$-algebra with involution of orthogonal type $(A,\ast)$, let us write $\delta(A,\ast) \in k^\times/(k^\times)^2$ for its discriminant (when $\mathrm{deg}(A)$ is even), $C_0(A,\ast)$ for its even Clifford $k$-algebra, and $\mathrm{Iv}(A,\ast)\subset \mathrm{SB}(A)$ for the associated involution variety; consult \cite[\S II]{Book-inv} \cite{Tao}. Recall from \cite[Thm.~8.10]{Book-inv} that when $\mathrm{deg}(A)$ is odd, $C_0(A,\ast)$ is a central simple $k$-algebra; that when $\mathrm{deg}(A)$ is even and $\delta(A,\ast)\notin (k^\times)^2$, $C_0(A,\ast)$ is a central simple algebra over its center $k(\sqrt{\delta(A,\ast)})$; and that when $\mathrm{deg}(A)$ is even and $\delta(A,\ast)\in (k^\times)^2$, $C_0(A,\ast)\simeq C_0(A,\ast)^+\times C_0(A,\ast)^-$ is a product of two central simple $k$-algebras. Recall also that $\mathrm{dim}_k(C_0(A,\ast))=2^{\mathrm{deg}(A)-1}$ and $\mathrm{dim}(\mathrm{Iv}(A,\ast))=\mathrm{deg}(A)-2$. Finally, when $k$ is formally-real, we will write $\mathrm{sgn}_P(A,\ast) \in \bbN$ for the signature of $(A,\ast)$ with respect to an ordering $P$ of $k$; consult \cite[\S 11]{Book-inv}.

\begin{example}[Quadrics]\label{ex:1}
In the particular case where $A$ is split, the central simple $k$-algebra with involution of orthogonal type $(A,\ast)$ becomes isomorphic to $(\mathrm{M}_{\mathrm{deg}(A)\times \mathrm{deg}(A)}(k), \ast_q)$, where $\ast_q$ is the adjoint involution of a uniquely determined quadratic form $q$ (up to similarity). Hence, in this particular case, the involution variety $\mathrm{Iv}(A,\ast)$ reduces to the quadric $Q_q$. Moreover, $\mathrm{deg}(A)=\mathrm{dim}(q)$, $\delta(A,
\ast)=\delta(q)$ and $C_0(A,\ast)=C_0(q)$. Furthermore, when $k$ is formally-real, $\mathrm{sgn}_P(A,\ast)$ reduces to $|\mathrm{sgn}_P(q)|$.
\end{example}
\begin{example}[Odd dimensional involution varieties]\label{ex:2}
Given a central simple $k$-algebra $A$ of odd degree, it is well-known that $A$ admits an involution of orthogonal type if and only if $A$ is split. Thanks to Example \ref{ex:1}, this shows that the odd dimensional involution varieties are the odd dimensional quadrics.
\end{example}
\begin{example}[Forms of quadrics]\label{ex:3}
Following Example \ref{ex:1}, note that every involution variety $\mathrm{Iv}(A,\ast)$ becomes isomorphic to a quadric after extension of scalars to a splitting field of $A$. Hence, involution varieties may be understood as ``forms of quadrics''. Moreover, as explained in \cite[\S2]{Tao}, the involution varieties admit the following characterization: a smooth projective $k$-variety $X$ of dimension $n-2$ is an involution variety if and only if $X\times_k \overline{k}$ is isomorphic to the (unique) quadric $(x_1^2 + \cdots + x^2_n=0)\subset \bbP_{\overline{k}}^{n-1}$. Furthermore, two involution varieties $\mathrm{Iv}(A,\ast)$ and $\mathrm{Iv}(A',\ast')$ are isomorphic if and only if the central simple $k$-algebras with involution of orthogonal type $(A,\ast)$ and $(A',\ast')$ are isomorphic.
\end{example}

\begin{example}[Products of two conics]\label{ex:4}
When $\mathrm{deg}(A)=4$ and $\delta(A,\ast)\in (k^\times)^2$, we have an isomorphism $(A,\ast) \simeq (Q_1, \ast_1) \otimes (Q_2, \ast_2)$, where $Q_1=(a_1,b_1)$ and $Q_2=(a_2,b_2)$ are uniquely determined quaternion $k$-algebras (up to isomorphism) and $\ast_1$ and $\ast_2$ are the conjugation involutions of $Q_1$ and $Q_2$, respectively. Following \cite[Thm.~4.15]{Tao}, this leads to an isomorphism $\mathrm{Iv}(A,\ast) \simeq C(a_1, b_1) \times C(a_2, b_2)$. This shows that the involution varieties of dimension $2$ with trivial discriminant are the products of two conics. Note that in the particular case where $A$ is split, we have $Q_1=Q_2=(a,b)$. Consequently, we conclude from Remark \ref{rk:quaternions} that $(x^2 - ay^2 - bu^2 + (ab)w^2=0) \simeq C(a,b) \times C(a,b)$ for every quaternion $k$-algebra $(a,b)$.
\end{example}
Example \ref{ex:4} motivates the following notation:
\begin{notation}[Condition $(\star)$]
Let $(A,\ast)$ be a central simple $k$-algebra with involution of orthogonal type, with $\mathrm{deg}(A)=4$ and $\delta(A,\ast) \in (k^\times)^2$. We will say that $(A,\ast)$ satisfies {\em condition $(\star)$} if $A$ is split. 
\end{notation}
\begin{theorem}\label{thm:main-1}
Let $(A,\ast)$ and $(A',\ast')$ be two central simple $k$-algebras with involutions of orthogonal type. If $[\mathrm{Iv}(A,\ast)]=[\mathrm{Iv}(A',\ast')]$, then the following holds:
\begin{itemize}
\item[(i)] We have $\mathrm{deg}(A)=\mathrm{deg}(A')$.
\item[(ii)] We have $(\delta(A,\ast)\in (k^\times)^2) \Leftrightarrow (\delta(A',\ast')\in (k^\times)^2)$.
\item[(iii)] We have $C_0(A,\ast) \simeq C_0(A',\ast')$.
\item[(iv)] We have $A \simeq A'$
\item[(v)] When $k$ is formally-real, we have $\mathrm{sgn}_P(A,\ast)=\mathrm{sgn}_P(A',\ast')$ for every ordering $P$ of $k$.
\end{itemize}
In the particular case where $\mathrm{deg}(A)=4$ and $\delta(A,\ast)\in (k^\times)^2$ (or, equivalently, $\mathrm{deg}(A')=4$ and $\delta(A',\ast')\in (k^\times)^2$), we assume moreover in items (iii)-(v) that $(A,\ast)$ and $(A',\ast')$ satisfy condition $(\star)$.
\end{theorem}
\begin{remark}
Note that if $C_0(A,\ast) \simeq C_0(A',\ast)$, then $\delta(A,\ast)=\delta(A',\ast')$.
\end{remark}
\begin{remark}[Severi-Brauer varieties]
It is well-known that two central simple $k$-algebras $A$ and $A'$ are isomorphic if and only if the associated Severi-Brauer varieties $\mathrm{SB}(A)$ and $\mathrm{SB}(A')$ are isomorphic. Consequently, item (iv) of Theorem \ref{thm:main-1} shows that if two involution varieties $\mathrm{Iv}(A,\ast) \subset \mathrm{SB}(A)$ and $\mathrm{Iv}(A',\ast') \subset \mathrm{SB}(A')$ have the same Grothendieck class, then the corresponding ``ambient'' Severi-Brauer varieties $\mathrm{SB}(A)$ and $\mathrm{SB}(A')$ are necessarily isomorphic!
\end{remark}
Theorem \ref{thm:main-1} enables the following applications:
\begin{theorem}\label{thm:applications-1}
Let $(A,\ast)$ and $(A',\ast')$ be two central simple $k$-algebras with involutions of orthogonal type. If $[\mathrm{Iv}(A,\ast)]=[\mathrm{Iv}(A',\ast')]$, then the following holds:
\begin{itemize}
\item[(i)] When $\mathrm{deg}(A)\leq4$ (or, equivalently, $\mathrm{deg}(A')\leq 4$), we have $\mathrm{Iv}(A,\ast)\simeq \mathrm{Iv}(A',\ast')$.
\item[(ii)] When $I(k)^3$ is torsion-free, we have $\mathrm{Iv}(A,\ast)\simeq \mathrm{Iv}(A',\ast')$. In the case where $k$ is formally-real and $\mathrm{deg}(A)$ is even (or, equivalently, $\mathrm{deg}(A')$ is even), we assume moreover that $\tilde{u}(k)< \infty$.
\end{itemize}
In the particular case where $\mathrm{deg}(A)=4$ and $\delta(A,\ast)\in (k^\times)^2$ (or, equivalently, $\mathrm{deg}(A')=4$ and $\delta(A',\ast')\in (k^\times)^2$), we assume moreover in items (i)-(ii) that $(A,\ast)$ and $(A',\ast')$ satisfy condition $(\star)$.
\end{theorem}
Note that Theorems \ref{thm:main-1} and \ref{thm:applications-1} are far reaching generalizations of Theorems \ref{thm:main} and \ref{thm:applications}, respectively. Consequently, the above inclusions \eqref{eq:inclusion} and \eqref{eq:inclusion2} admit extensions to the realm of involution varieties. In particular, the above Example \ref{ex:adic} admits the following extension:
\begin{example}[Involution varieties of dimension $\leq 2$ over $\bbQ_p$]\label{ex:inv-2}
As mentioned in Example \ref{ex:2}, there is no difference between involution varieties of dimension $1$ and quadrics of dimension $1$. Hence, we address solely the $2$-dimensional case. First, recall from \cite[\S XII-XIII]{Serre2} that given any local field $l$, there is a unique division quaternion $l$-algebra ${\bf D}$ up to isomorphism; when $l=\bbQ_p$, we have ${\bf D}\simeq(\epsilon,p)$. As explained in \cite[\S15.B]{Book-inv}, the assignment $(A,\ast) \mapsto C_0(A,\ast)$ gives rise to a bijection between the set of central simple $\bbQ_p$-algebras of degree $4$ with involution of orthogonal type up to isomorphism and the set of quaternion algebras over some \'etale quadratic extension of $\bbQ_p$ up to isomorphism. The inverse bijection is induced by the assignment $Q \mapsto (N_{l/\bbQ_p}(Q), N_{l/\bbQ_p}(\ast))$, where $l$ is the center of $Q$, $\ast$ is the conjugation involution of $Q$, and $N_{l/k}(-)$ is the norm construction. Consequently, since $\bbQ^\times_p/(\bbQ_p^\times)^2= \{1, \epsilon, p, \epsilon p\}$, there are  nine involution varieties of dimension $2$ over $\bbQ_p$ up to isomorphism. Note that under the assignment $\mathrm{Iv}(A,\ast) \mapsto (A,\ast) \mapsto C_0(A,\ast)$, the involution varieties \eqref{eq:quadric1}-\eqref{eq:quadric3} correspond to the following quaternion algebras:
\begin{eqnarray*}
(1,1) \times (1,1) \,\,\mathrm{over}\,\,\bbQ_p \times \bbQ_p &&  (\epsilon,p) \times (\epsilon,p) \,\,\mathrm{over}\,\,\bbQ_p \times \bbQ_p \\
(1,1)\,\,\mathrm{over}\,\,\bbQ_p(\sqrt{\epsilon}) &&  (1,1)\,\,\mathrm{over}\,\,\bbQ_p(\sqrt{p}) \\
(1,1)\,\,\mathrm{over}\,\,\bbQ_p(\sqrt{\epsilon p}) &&  {\bf D} \,\,\mathrm{over}\,\,\bbQ_p(\sqrt{\epsilon p})\,.
\end{eqnarray*}
Therefore, in addition to \eqref{eq:quadric1}-\eqref{eq:quadric3}, we can also consider the following two involution varieties:
\begin{eqnarray}\label{eq:involutions}
\mathrm{Iv}(N_{\bbQ_p(\sqrt{\epsilon})/\bbQ_p}({\bf D}), N_{\bbQ_p(\sqrt{\epsilon})/\bbQ_p}(\ast)) && \mathrm{Iv}(N_{\bbQ_p(\sqrt{p})/\bbQ_p}({\bf D}), N_{\bbQ_p(\sqrt{p})/\bbQ_p}(\ast))\,.
\end{eqnarray}
Making use of item (i) of Theorem \ref{thm:applications-1}, we hence conclude that the Grothendieck classes of the eight involution varieties, namely \eqref{eq:quadric1}-\eqref{eq:quadric3} and \eqref{eq:involutions}, remain distinct in $K_0\mathrm{Var}(\bbQ_p)$.
\end{example}
The following result of Koll{\'a}r\footnote{Consult also the subsequent work \cite{Hogadi}.} \cite{Kollar} explains why condition $(\star)$ is needed in Theorems \ref{thm:main-1} and \ref{thm:applications-1}:
\begin{theorem}[Products of two conics]\label{thm:Kollar}
Let $Q_1$ and $Q_2$, resp. $Q'_1$ and $Q'_2$, be quaternion $k$-algebras and $C(Q_1)$ and $C(Q_2)$, resp. $C(Q'_1)$ and $C(Q'_2)$, the associated conics. The following conditions are equivalent:
\begin{itemize}
\item[(a)] $[C(Q_1) \times C(Q_2)]= [C(Q'_1)\times C(Q'_2)]$
\item[(b)] $C(Q_1) \times C(Q_2) \,\,\mathrm{is}\,\,\mathrm{birational}\,\,\mathrm{to}\,\,C(Q'_1) \times C(Q'_2)$
\item[(c)] $\langle [Q_1],[Q_2]\rangle = \langle [Q'_1],[Q'_2]\rangle$ in the Brauer group $\mathrm{Br}(k)$.
\end{itemize}
\end{theorem}
\begin{example}[Products of two conics over $\bbR$]\label{ex:conicsR}
Let $k=\bbR$, $Q_1={\bf H}$, $Q_2=(1,1)$, and $Q'_1=Q'_2={\bf H}$. Following Example \ref{ex:4}, let $(A,\ast):=(Q_1,\ast_1) \otimes (Q_2,\ast_2)$ and $(A',\ast'):=(Q'_1,\ast'_1) \otimes (Q'_2,\ast'_2)$. Note that $(A',\ast') \simeq (\mathrm{M}_{4\times 4}(\bbR), \ast'_{q'})$, where $\ast'_{q'}$ is the adjoint involution of the quadratic form $q'=\langle 1, 1, 1, 1\rangle$. On the one hand, since $\langle [{\bf H}], [(1,1)]\rangle=\langle [{\bf H}],[{\bf H}]\rangle$ in the Brauer group $\mathrm{Br}(\bbR)$, Theorem \ref{thm:Kollar} implies that $[\mathrm{Iv}(A,\ast)]=[C({\bf H})\times \bbP^1]=[C({\bf H})\times C({\bf H})]=[\mathrm{Iv}(A',\ast')]$. On the other hand, the even Clifford algebra $C_0(A,\ast) \simeq {\bf H} \times (1,1)$ is {\em not} isomorphic to $C_0(A',\ast')\simeq {\bf H} \times {\bf H}$, the central simple $\bbR$-algebra $A\simeq \mathrm{M}_{2\times 2}({\bf H})$ is {\em not} isomorphic to $A'\simeq \mathrm{M}_{4\times 4}(\bbR)$, the signature $\mathrm{sgn}(A,\ast)=0$ is {\em different} from $\mathrm{sgn}(A',\ast') = |\mathrm{sgn}(\langle 1, 1, 1, 1\rangle)|= 4$, and the involution variety $C({\bf H}) \times \bbP^1$ is {\em not} isomorphic to $C({\bf H}) \times C({\bf H})$. This shows, in particular, that the above Theorems \ref{thm:main-1} and \ref{thm:applications-1} are {\em false} without assuming condition $(\star)$.
\end{example}
Finally, note that by combining Koll\'ar's Theorem \ref{thm:Kollar} with item (i) of Theorem \ref{thm:applications-1}, we obtain a complete description of the Grothendieck classes of {\em all} the involution varieties of dimension $2$ (over any base field $k$). Here is one illustrative example:
\begin{example}[Grothendieck classes of the involution varieties of dimension $2$ over $\bbQ_p$]
As explained in Example \ref{ex:inv-2}, there are nine involution varieties of dimension $2$ over $\bbQ_p$, with $p\neq 2$, up to isomorphism. Eight of them were described in Examples \ref{ex:adic} and \ref{ex:inv-2}, namely \eqref{eq:quadric1}-\eqref{eq:quadric3} and \eqref{eq:involutions}, and these have distinct Grothendieck classes. The remaining involution variety (which has trivial discriminant) is the following one $\mathrm{Iv}(((\epsilon,p),\ast)\otimes ((1,1),\ast) \simeq C(\epsilon,p))\times C(1,1) \simeq C(\epsilon, 1) \times \bbP^1$. Following Example \ref{ex:4}, note that the right-hand side of \eqref{eq:quadric1} is isomorphic to the product $C(\epsilon, p) \times C(\epsilon, p)$. Since $\langle [(\epsilon,p)], [(\epsilon,p)]\rangle=\langle [(\epsilon,p)],[(1,1)]\rangle$ in the Brauer group $\mathrm{Br}(\bbQ_p)$, Theorem \ref{thm:Kollar} implies that $[C(\epsilon,p)\times C(\epsilon,p)]=[C(\epsilon,p)\times \bbP^1]$. Consequently, we conclude that the nine involution varieties of dimension $2$ give rise to eight distinct Grothendieck classes. 
\end{example}
\section{Preliminaries}\label{sec:preliminaries}
Throughout the article $k$ denotes a base field of characteristic zero, and we will write $G:=\mathrm{Gal}(\overline{k}/k)$ for its absolute Galois group. Given a central simple $k$-algebra $A$, we will write $\mathrm{ind}(A)$ for its index and $[A]$ for its class in the Brauer group $\mathrm{Br}(k)$. Recall that $\mathrm{Br}(k) \simeq \oplus_{p\,\mathrm{prime}} \mathrm{Br}(k)\{p\}$, where $\mathrm{Br}(k)\{p\}$ stands for the $p$-power torsion subgroup. Finally, in order to simplify the exposition, we will write $\mathrm{pt}$ instead of $\mathrm{Spec}(k)$.
\subsection{Dg categories}\label{sub:dg}
Throughout the article, we will assume some basic familiarity with the language of dg categories; consult, for example, Keller's survey \cite{ICM-Keller}. We will write $\dgcat(k)$ for the category of (small) dg categories and $\dgcat_{\mathrm{sp}}(k)$ for the full subcategory of smooth proper dg categories in the sense of Kontsevich \cite{Miami,finMot,IAS}. Examples of smooth proper dg categories include, for example, the finite-dimensional $k$-algebras of finite global dimension $A$ as well as the dg categories of perfect complexes $\perf_\dg(X)$ associated to smooth proper $k$-schemes $X$. As explained in \cite[\S1.7]{book}, the symmetric monoidal category $(\dgcat_{\mathrm{sp}}(k), \otimes)$ is rigid\footnote{Recall that a symmetric monoidal category is called {\em rigid} if all its objects are dualizable.}, with the dual of a smooth proper dg category $\cA$ being its opposite dg category $\cA^\op$.
\subsection{Chow motives}\label{sub:motives}
Given a commutative ring of coefficients $R$, recall from \cite[\S4.1]{Andre}\cite{Manin1} the definition of the category of {\em Chow motives} $\Chow(k)_R$. This category is $R$-linear, rigid symmetric monoidal and idempotent complete. Moreover, it comes equipped with a symmetric monoidal functor $\mathfrak{h}(-)_R \colon \mathrm{SmProj}(k) \to \Chow(k)_R$ defined on the category of smooth projective $k$-schemes. The Chow motive $\mathfrak{h}(\bbP^1)_R$ of the projective line $\bbP^1$ decomposes into a direct sum $\mathfrak{h}(\mathrm{pt})_R \oplus R(-1)$. The direct summand $R(-1)$ is called the {\em Lefschetz motive} and its $\otimes$-inverse $R(1)$ the {\em Tate motive}. Given a smooth projective $k$-scheme $X$ and an integer $i \in \mathbb{Z}$, we will write $\mathfrak{h}(X)_R(i)$ instead of $\mathfrak{h}(X)_R\otimes R(1)^{\otimes i}$. Under these notations, given two smooth projective $k$-schemes $X$ and $Y$ and two integers $i, j \in \bbZ$, we have an isomorphism
$$ \mathrm{Hom}_{\Chow(k)_R}(\mathfrak{h}(X)_R(i),\mathfrak{h}(Y)_R(j))\simeq \cZ^{\mathrm{dim}(X)-i+j}_{\sim \mathrm{rat}}(X\times Y)_R\,,$$
where the right-hand side stands for the $R$-module of algebraic cycles on $X\times Y$ of codimension $\mathrm{dim}(X) - i + j$ up to rational equivalence. Finally, in the particular case where $R=\bbZ$, we will write $\Chow(k)$ instead of $\Chow(k)_\bbZ$ and $\mathfrak{h}(-)$ instead of $\mathfrak{h}(-)_\bbZ$.

\subsection{Noncommutative Chow motives}\label{sub:NCmotives}
Given a commutative ring of coefficients $R$, recall from \cite[\S4.1]{book} the definition of the category of {\em noncommutative Chow motives} $\NChow(k)_R$. This category is $R$-linear, rigid symmetric monoidal, idempotent complete, and comes equipped with a symmetric monoidal functor $U(-)_R \colon \dgcat(k)_{\mathrm{sp}} \to \NChow(k)_R$. Given smooth proper dg categories $\cA$ and $\cA'$, we have an isomorphism 
$$\Hom_{\NChow(k)_R}(U(\cA)_R, U(\cA')_R)\simeq K_0(\cA^\op \otimes \cA')_R\,,$$
where the right-hand side stands for the $R$-linearized Grothendieck group of $\cA^\op \otimes \cA'$. Moreover, the composition law on $\NChow(k)$ is induced by the (derived) tensor product of bimodules, and the identity of $U(\cA)_R$ is the Grothendieck class of the diagonal $\cA\text{-}\cA$-bimodule $\cA$. In the particular case where $R=\bbZ$, we will write $\NChow(k)$ instead of $\NChow(k)_\bbZ$ and $U(-)$ instead of $U(-)_\bbZ$.
\begin{theorem}{(see \cite[Thm.~9.1]{Homogeneous})}\label{thm:Brauer}
Given two central simple $k$-algebras $A$ and $A'$, we have the equivalence:
\begin{eqnarray*}
U(A) \simeq U(A')\,\,\mathrm{in}\,\,\NChow(k) & \Leftrightarrow & [A]=[A'] \,\,\mathrm{in}\,\,\mathrm{Br}(k)\,.
\end{eqnarray*}
\end{theorem}
\begin{theorem}{(see \cite[Thm.~2.20(iv)]{Separable})}\label{thm:Brauer1}
Given two families of central simple $k$-algebras $\{A_i\}_{1\leq i \leq n}$ and $\{A'_j\}_{1\leq j \leq m}$, we have an isomorphism $\oplus_{i=1}^n U(A_i) \simeq \oplus_{j=1}^m U(A_j)$ in $\NChow(k)$ if and only if $n=m$ and for every prime number $p$ there exists a permutation $\sigma_p$ such that $[A'_i]^p=[A_{\sigma_p(i)}]^p$ in $\mathrm{Br}(k)\{p\}$ for every $i$.
\end{theorem}
\begin{remark}[Central simple algebras over field extensions]\label{ex:key}
Let $l/k$, resp. $l'/k$, be a finite Galois field extension and $A$, resp. $A'$, a central simple $l$-algebra, resp. central simple $l'$-algebra. Let us denote by $H\subseteq G$, resp. $H'\subseteq G$, the (unique) subgroup such that $\overline{k}^H=l$, resp. $\overline{k}^{H'}=l'$. Given a commutative ring of coefficients $R$, recall from \cite[Thm.~2.13]{Separable} that $\Hom_{\NChow(k)_R}(U(A)_R, U(A')_R)$ can be identified with the $R$-module $\mathrm{Map}^G(G/H\times G/H', R)$ of $G$-invariant maps from the finite $G$-set $G/H\times G/H'$ (equipped with the diagonal $G$-action) to $R$. Under these identifications, the composition map
$$ \Hom_{\NChow(k)_R}(U(A)_R,U(A')_R) \times  \Hom_{\NChow(k)_R}(U(A')_R,U(A)_R) \too \Hom_{\NChow(k)_R}(U(A)_R,U(A)_R)$$
corresponds to the bilinear pairing
$$ \mathrm{Map}^G(G/H \times G/H', R) \times \mathrm{Map}^G(G/H' \times G/H, R) \too \mathrm{Map}^G(G/H \times G/H, R)$$
that sends a pair of $G$-invariant maps $(\alpha,\beta)$ to the following $G$-invariant map:
$$ (\overline{g_1},\overline{g_3}) \mapsto \sum_{\overline{g_2} \in G/H'}\alpha(\overline{g_1}, \overline{g_2}) \beta(\overline{g_2},\overline{g_3})\mathrm{ind}\big((A\otimes_{\overline{k}^H}\overline{k}^{(H \cap H')})^\op\otimes_{\overline{k}^{(H \cap H')}}(A'\otimes_{\overline{k}^{H'}}\overline{k}^{(H \cap H')})\big)^2\,.$$ 
Moreover, the identity of $U(A)_R$ corresponds to the $G$-invariant map $G/H \times G/H \to R$ with $1$ in the diagonal and $0$ elsewhere.
\end{remark}
\subsection{Numerical motives}
Given an additive rigid symmetric monoidal category $\cC$, recall from \cite[\S7]{Kahn} that its {\em $\cN$-ideal} is defined as follows
$$ \cN(a,b):=\{f \in \Hom_\cC(a,b)\,|\, \forall g \in \Hom_\cC(b,a)\,\,\mathrm{we}\,\,\mathrm{have}\,\,\mathrm{tr}(g\circ f)=0 \}\,,$$
where $\mathrm{tr}(g\circ f)$ is the categorical trace of $g\circ f$. Via the adjunction isomorphism $\Hom_\cC(a,b)\simeq \Hom_\cC({\bf 1}, a^\vee\otimes b)$, where ${\bf 1}$ stands for the $\otimes$-unit and $a^\vee$ for the dual of $a$, the $\cN$-ideal can also be described as follows:
\begin{equation*}
\cN(a,b)=\{f \in \Hom_\cC({\bf 1},a^\vee \otimes b)\,|\, \forall g \in \Hom_\cC(a^\vee \otimes b,{\bf 1})\,\,\mathrm{we}\,\,\mathrm{have}\,\,g\circ f=0 \}\,.
\end{equation*}
Given a commutative ring of coefficients $R$, recall from \cite[\S4.1]{Andre}\cite{Manin1} that the category of {\em numerical motives} $\Num(k)_R$ is defined as the idempotent completion of the quotient of $\Chow(k)_R$ by the $\otimes$-ideal $\cN$. This category is $R$-linear, rigid symmetric monoidal and idempotent complete. Moreover, given two smooth projective $k$-schemes $X$ and $Y$ and two integers $i, j \in \bbZ$, we have an isomorphism
$$ \mathrm{Hom}_{\mathrm{Num}(k)_R}(\mathfrak{h}(X)_R(i),\mathfrak{h}(Y)_R(j)) \simeq \cZ_{\sim \mathrm{num}}^{\mathrm{dim}(X) - i + j}(X\times Y)_R\,,$$
where the right-hand side stands for the $R$-module of algebraic cycles up to numerical equivalence. 
\subsection{Noncommutative numerical motives}\label{sub:numerical}
Given a commutative ring of coefficients $R$, recall from \cite[\S4.6]{book} that the category of {\em noncommutative numerical motives} $\NNum(k)_R$ is defined as the idempotent completion of the quotient of $\NChow(k)_R$ by the $\otimes$-ideal $\cN$. This category  is $R$-linear, rigid symmetric monoidal and idempotent complete. 
\subsection{Noncommutative radical motives}\label{sub:radical}
Given an additive category $\cC$, its {\em $\cR$-ideal} is defined as follows:
\begin{equation*}
\cR(a,b):= \{f\in \Hom_\cC(a,b)\,|\, \forall g \in \Hom_\cC(b,a)\,\,\mathrm{the}\,\,\mathrm{endomorphism}\,\,\id_a - g\circ f\,\,\mathrm{is}\,\,\mathrm{invertible} \}\,.\end{equation*}
Given a commutative ring of coefficients $R$, the category of {\em noncommutative radical motives} $\mathrm{NRad}(k)_R$ is defined as the idempotent completion of the quotient of $\NChow(k)_R$ by the ideal $\cR$. By construction, this category is $R$-linear, additive, and idempotent complete. 
\subsection{Motivic measures}
Let $K_0(\Num(k))$ and $K_0(\NChow(k))$ be the Grothendieck rings of the additive symmetric monoidal categories $\Num(k)$ and $\NChow(k)$, respectively. The following motivic measures will be used throughout the article:
\begin{proposition}{(see \cite[Cor.~13.2.2.1]{Andre})}\label{prop:measure1}
The assignment $X \mapsto \mathfrak{h}(X)$, with $X$ a smooth projective $k$-scheme, gives rise to a motivic measure $\mu_{\mathrm{c}}\colon K_0\mathrm{Var}(k) \to K_0(\Num(k))$.
\end{proposition}
\begin{proposition}{(see \cite[Prop.~4.1]{Tits})}\label{prop:measure}
The assignment $X \mapsto U(\perf_\dg(X))$, with $X$ a smooth projective $k$-scheme, gives rise to a motivic measure $\mu_{\mathrm{nc}}\colon K_0\mathrm{Var}(k) \to K_0(\NChow(k))$.
\end{proposition}
\section{Cancellation property}\label{sec:cancellation}
We start by recalling the following cancellation result:
\begin{proposition}{(see \cite[Prop.~4.9]{Tits})}\label{prop:cancellation-0}
Let $\{A_i\}_{1\leq i \leq n}$ and $\{A'_j\}_{1\leq j \leq m}$ be two families of central simple $k$-algebras. Given any noncommutative Chow motive $N\!\!M \in \NChow(k)$, we have the following implication:
\begin{eqnarray}
N\!\!M \oplus \oplus_{i=1}^n U(A_i) \simeq N\!\!M \oplus \oplus_{j=1}^m U(A'_j) & \Rightarrow & n=m\,\,\,\mathrm{and}\,\, \oplus_{i=1}^n U(A_i) \simeq \oplus_{j=1}^m U(A'_j)\,.
\end{eqnarray}
\end{proposition}
The following result, which is of independent interest, will play a key role in the sequel:
\begin{theorem}[Cancellation]\label{thm:cancellation}
Let $l/k$, resp. $l'/k$, be a field a extension of degree $2$ and $A$, resp. $A'$, a central simple $l$-algebra, resp. central simple $l'$-algebra, such that $\mathrm{ind}(A)$, resp. $\mathrm{ind}(A')$, is a power of $2$. Moreover, let $B$ and $B'$ be two central simple $k$-algebras. Under these assumptions, given any noncommutative Chow motive $N\!\!M \in \NChow(k)$ and integer $n\geq0$, we have the following implication:
\begin{eqnarray*}
N\!\!M\oplus U(B)^{\oplus n} \oplus U(A)  \simeq N\!\!M\oplus U(B')^{\oplus n} \oplus U(A')& \Rightarrow& U(A) \simeq U(A')\,. 
\end{eqnarray*}
\end{theorem}
In order to prove Theorem \ref{thm:cancellation}, we need several ingredients.
\begin{lemma}\label{lem:key1}
Given two field extensions $l/k$ and $l'/k$ of degree $2$, we have the following equivalences:
\begin{equation*}\label{eq:equivalences}
l\simeq l' \quad \Leftrightarrow \quad U(l)_\bbQ \simeq U(l')_\bbQ \,\,\mathrm{in}\,\,\NChow(k)_\bbQ \quad  \Leftrightarrow \quad U(l)_\bbQ \simeq U(l')_\bbQ \,\,\mathrm{in}\,\,\NNum(k)_\bbQ\,.
\end{equation*}
\end{lemma}
\begin{proof}
Both implications $(\Rightarrow)$ are clear. We start by proving the right-hand side implication $(\Leftarrow)$. Note that since the field extension $l/k$ is of degree $2$, the composition map
$$ \Hom_{\NChow(k)_\bbQ}(U(k)_\bbQ, U(l)_\bbQ) \times \Hom_{\NChow(k)_\bbQ}(U(l)_\bbQ, U(k)_\bbQ) \too \Hom_{\NChow(k)_\bbQ}(U(k)_\bbQ, U(k)_\bbQ)$$
corresponds to the bilinear pairing $\bbQ \times \bbQ \to \bbQ, (\alpha,\beta) \mapsto 2\alpha\beta$. Note also that since the field extension $l/k$ is Galois, we have the following isomorphism of $k$-algebras $l\otimes_k l \stackrel{\simeq}{\to} l \times l, \lambda_1 \otimes \lambda_2 \mapsto (\lambda_1\lambda_2, \sigma(\lambda_1)\lambda_2)$, where $\sigma$ stands for the generator of the Galois group of $l/k$. This implies that 
$$U(l)^\vee_\bbQ \otimes U(l)_\bbQ \simeq U(l^\op)_\bbQ \otimes U(l)_\bbQ \simeq U(l^\op \otimes_k l)_\bbQ = U(l\otimes_k l)_\bbQ \simeq U(l)_\bbQ \oplus U(l)_\bbQ\,.$$ 
Therefore, following the definition of the category of noncommutative numerical motives, we observe that $\mathrm{End}_{\NChow(k)_\bbQ}(U(l)_\bbQ) = \mathrm{End}_{\NNum(k)_\bbQ}(U(l)_\bbQ)$. All the above holds {\em mutatis mutandis} with $l$ replaced by $l'$. Hence, we conclude that if $U(l)_\bbQ \simeq U(l')_\bbQ \,\,\mathrm{in}\,\,\NNum(k)_\bbQ$, then $U(l)_\bbQ \simeq U(l')_\bbQ \,\,\mathrm{in}\,\,\NChow(k)_\bbQ$.

We now prove the left-hand side implication $(\Leftarrow)$. Let $H \subseteq G$ and $H'\subseteq G$ be the subgroups of index $2$ such that $\overline{k}^H=l$ and $\overline{k}^{H'}=l'$, respectively. Recall from Remark \ref{ex:key} that $\Hom_{\NChow(k)_\bbQ}(U(l)_\bbQ, U(l')_\bbQ)$ can be identified with the $\bbQ$-vector space $\mathrm{Map}^G(G/H \times G/H', \bbQ)$ of $G$-invariant maps from the finite $G$-set $G/H\times G/H'$ to $\bbQ$. Recall also that the composition map
\begin{equation*}
\Hom_{\NChow(k)_\bbQ}(U(l)_\bbQ, U(l')_\bbQ) \times \Hom_{\NChow(k)_\bbQ}(U(l')_\bbQ, U(l)_\bbQ) \too \Hom_{\NChow(k)_\bbQ}(U(l)_\bbQ, U(l)_\bbQ)
\end{equation*}
corresponds to the bilinear pairing
$$
\mathrm{Map}^G(G/H \times G/H', \bbQ) \times \mathrm{Map}^G(G/H' \times G/H, \bbQ) \too \mathrm{Map}^G(G/H \times G/H, \bbQ)
$$
that sends $(\alpha, \beta)$ to the $G$-invariant map $(\overline{g_1}, \overline{g_3}) \mapsto \sum_{\overline{g_2} \in G/H'} \alpha(\overline{g_1},\overline{g_2}) \beta(\overline{g_2},\overline{g_3})$. Moreover, the identity of $U(l)_\bbQ$ corresponds to the $G$-invariant map $G/H \times G/H \to \bbQ$ with $1$ in the diagonal and $0$ elsewhere. This shows, in particular, that the $\bbQ$-algebra of endomorphisms $\mathrm{End}_{\NChow(k)_\bbQ}(U(l)_\bbQ)$ corresponds to the group $\bbQ$-algebra $\bbQ[G/H]$. Now, let us assume that $l\not\simeq l'$, i.e., that the $k$-algebras $l$ and $l'$ are not isomorphic. Thanks to Galois theory, we have $H\neq H'$. Moreover, since the field extensions $l/k$ and $l'/k$ are of degree $2$, we have $H\not\subseteq H'$ and $H'\not\subseteq H$. This implies that the (diagonal) $G$-action on the set $G/H \times G/H'$ is transitive. Therefore, it follows from the above description of the category of noncommutative Chow motives that $U(l)_\bbQ \simeq U(l')_\bbQ$ in $\NChow(k)_\bbQ$ if and only if there exist rational numbers $\alpha, \beta \in \bbQ$ such that 
\begin{equation}\label{eq:system-equations}
\begin{cases}\alpha \beta + \alpha \beta =1\\ \alpha \beta + \alpha  \beta =0 \end{cases}\,.
\end{equation}
Since the system of equations \eqref{eq:system-equations} is impossible, we hence conclude that $U(l)_\bbQ \not\simeq U(l')_\bbQ$ in $\NChow(k)_\bbQ$.
\end{proof}
\begin{lemma}\label{lem:key2}
Given a field extension $l/k$ of degree $2$ and two central simple $l$-algebras $A$ and $A'$, we have the following equivalence:
\begin{eqnarray*}
U(A) \simeq U(A') \,\,\mathrm{in}\,\,\NChow(k) & \Leftrightarrow & [A]= [A'] \,\,\mathrm{in}\,\,\mathrm{Br}(l) \,.
\end{eqnarray*}
\end{lemma}
\begin{proof}
Let $H \subseteq G$ be the subgroup of index $2$ such that $\overline{k}^H=l$. Given a commutative ring of coefficients $R$, recall from Remark \ref{ex:key} that $\Hom_{\NChow(k)_R}(U(A)_R, U(A')_R)$ can be identified with the $R$-module $\mathrm{Map}^G(G/H\times G/H, R)$ of $G$-invariant maps from the finite $G$-set $G/H \times G/H$ to $R$. Recall also that, under these identifications, the composition map
\begin{eqnarray*}
\Hom_{\NChow(k)_R}(U(A)_R, U(A')_R) \times \Hom_{\NChow(k)_R}(U(A')_R, U(A)_R) \too \Hom_{\NChow(k)_R}(U(A)_R, U(A)_R)
\end{eqnarray*}
corresponds to the bilinear pairing
$$ \mathrm{Map}^G(G/H \times G/H, R) \times \mathrm{Map}^G(G/H \times G/H, R) \too \mathrm{Map}^G(G/H \times G/H, R)$$
that sends $(\alpha, \beta)$ to the $G$-invariant map $(\overline{g_1}, \overline{g_3}) \mapsto \sum_{\overline{g_2} \in G/H} \alpha(\overline{g_1},\overline{g_2}) \beta(\overline{g_2},\overline{g_3}) \mathrm{ind}(A^\op \otimes_l A')^2$. Moreover, the identity of $U(A)_R$ corresponds to the $G$-invariant map $G/H \times G/H \to R$ with $1$ in the diagonal and $0$ elsewhere. This shows, in particular, that the $R$-algebra of endomorphisms $\mathrm{End}_{\NChow(k)_R}(U(A)_R)$ corresponds to the group $R$-algebra $R[G/H]$. Note that since the field extension $l/k$ is of degree $2$, the $G$-set $G/H \times G/H$ has two orbits (the elements in the diagonal and the elements outside the diagonal). Therefore, thanks to the above description of the category of noncommutative Chow motives, we observe that $U(A)_R\simeq U(A')_R$ in $\NChow(k)_R$ if and only if there exist elements $\alpha^+, \alpha^-\in R$ and $\beta^+, \beta^- \in R$ such that
\begin{equation}\label{eq:system-equations1}
\begin{cases} \alpha^+ \beta^+ \mathrm{ind}(A^\op \otimes_l A')^2 + \alpha^- \beta^- \mathrm{ind}(A^\op \otimes_l A')^2=1  \\  \alpha^+ \beta^- \mathrm{ind}(A^\op \otimes_l A')^2 + \alpha^- \beta^+ \mathrm{ind}(A^\op \otimes_l A')^2=0\end{cases} \Leftrightarrow \begin{cases} (\alpha^+ \beta^+ + \alpha^- \beta^-) \mathrm{ind}(A^\op \otimes_l A')^2 =1 \\ (\alpha^+ \beta^- + \alpha^- \beta^+) \mathrm{ind}(A^\op \otimes_l A')^2=0\,.\end{cases}
\end{equation}
Now, let us restrict ourselves to the case where $R=\bbZ$. Note that if the system of equations \eqref{eq:system-equations1} holds, then $\mathrm{ind}(A^\op \otimes_l A')=1$. Conversely, if $\mathrm{ind}(A^\op \otimes_l A')=1$, then the system of equations \eqref{eq:system-equations1} holds; take, for example, $\alpha^+=\beta^+=1$ and $\alpha^-=\beta^-=0$. Consequently, the proof follows now from the classical fact that $\mathrm{ind}(A^\op \otimes_l A')=1$ if and only if $[A]=[A']$ in $\mathrm{Br}(l)$.
\end{proof}
\begin{lemma}\label{lem:computation-B-A}
Let $B$ a central simple $k$-algebra, $l/k$ a field extension of degree $2$ and $A$ a central simple $l$-algebra. Under these assumptions, we have $\Hom_{\mathrm{NRad}(k)_{\bbF_2}}(U(B)_{\bbF_2},U(A)_{\bbF_2}) =0$. 
\end{lemma}
\begin{proof}
Let $H \subseteq G$ be the subgroup of index $2$ such that $\overline{k}^H=l$. Given a commutative ring of coefficients $R$, recall from Remark \ref{ex:key} that $\Hom_{\NChow(k)_R}(U(B)_R,U(A)_R)$ can be identified with the $R$-module $\mathrm{Map}^G(G/G \times G/H, R)$ of $G$-invariant maps from the finite set $G$-set $G/G\times G/H = G/H$ to $R$. Recall also that, under these identifications, the composition map
\begin{equation}\label{eq:composition-map}
\Hom_{\NChow(k)_R}(U(B)_R,U(A)_R) \times \Hom_{\NChow(k)_R}(U(A)_R,U(B)_R) \too \Hom_{\NChow(k)_R}(U(B)_R,U(B)_R)
\end{equation}
corresponds to the bilinear pairing $R \times R \to R, (\alpha, \beta) \mapsto 2\alpha\beta \mathrm{ind}((B\otimes_kl)^\op \otimes_l A)^2$. Consequently, in the particular case where $R=\bbF_2$, the composition map \eqref{eq:composition-map} is equal to zero. By definition of the category of noncommutative radical motives, this hence implies that $\Hom_{\mathrm{NRad}(k)_{\bbF_2}}(U(B)_{\bbF_2}, U(A)_{\bbF_2})=0$.
\end{proof}
\begin{lemma}\label{lem:key}
Let $l/k$ be a field extension of degree $2$ and $A$ a central simple $l$-algebra. Given any noncommutative Chow motive $N\!\!M\in \NChow(k)$, we have a direct sum decomposition $N\!\!M_{\bbF_2}\simeq M\!\!N \oplus U(A)_{\bbF_2}^{\oplus m}$ in the category $\mathrm{NRad}(k)_{\bbF_2}$  for some integer $m\geq 0$, where $M\!\!N$ is a noncommutative radical motive which does not contains $U(A)_{\bbF_2}$ as a direct summand.
\end{lemma}
\begin{proof}
Recall from the proof of Lemma \ref{lem:key2} that $\mathrm{End}_{\NChow(k)}(U(A))$ corresponds to $\bbZ[G/H]$, where $H \subseteq G$ is the subgroup of index $2$ such that $\overline{k}^H=l$. Note that since the field extension $l/k$ is of degree $2$, we have $\bbZ[G/H]=\{\alpha^+1+ \alpha^-\sigma\,|\,\alpha^+, \alpha^- \in \bbZ\}$, where $\sigma$ stands for the generator of the Galois group of $l/k$.

By definition, the category $\mathrm{NRad}(k)_{\bbF_2}$ is idempotent complete. Therefore, we can inductively split the direct summand $U(A)_{\bbF_2}$ from $N\!\!M_{\bbF_2}$. We claim that this inductive process stops. Let us suppose by absurd that it does not. If this is the case, then, given any integer $n \geq 1$, $U(A)_{\bbF_2}^{\oplus n}$ is a direct summand of $N\!\!M_{\bbF_2}$. Using the fact that the quotient functor $\NChow(k)_{\bbF_2} \to \mathrm{NRad}(k)_{\bbF_2}$ reflects (co)retractions (consult \cite[Prop.~1.4.4]{Kahn}), we conclude that $U(A)_{\bbF_2}^{\oplus n}$ is a direct summand of $N\!\!M_{\bbF_2}$ in the category $\NChow(k)_{\bbF_2}$. Hence, there exist morphisms $\rho\colon U(A)^{\oplus n} \to N\!\!M$ and $\varrho\colon N\!\!M \to U(A)^{\oplus n}$ in the category $\NChow(k)$ such that: 
\begin{eqnarray}\label{eq:relations}
\varrho\circ \rho = \begin{bmatrix}
\alpha^+_{11}  1 + \alpha^-_{11} \sigma & \cdots & \alpha^+_{1n}  1 + \alpha^-_{1n} \sigma\\
\vdots &  & \vdots \\
\alpha^+_{n1}  1 + \alpha^-_{n1} \sigma & \cdots & \alpha^+_{nn}  1 + \alpha^-_{nn} \sigma
\end{bmatrix}_{n\times n} & \mathrm{with} & \begin{cases} \alpha^+_{ij}\,\,\mathrm{odd} \quad i=j \\ \alpha^+_{ij}\,\,\mathrm{even}\,\,\,\, i\neq j\end{cases}\,\,\,\,\alpha^-_{ij}\,\,\mathrm{even}\,.
\end{eqnarray}

Recall from \cite[\S2.2.8]{book} that Hochschild homology gives rise to a functor $HH_0(-):\dgcat_{\mathrm{sp}}(k) \to \mathrm{vect}(k)$, with values in the category of finite-dimensional $k$-vector spaces, such that $HH_0(U(\cA))\simeq HH_0(\cA)$ for every smooth proper dg category $\cA$. As proved in \cite[Thm.~2.1]{Azumaya}, the canonical map $l \to A$ induces an isomorphism $HH_0(l) \stackrel{\simeq}{\to} HH_0(A)$. Moreover, it is well-known that $HH_0(l)\simeq l/[l,l]\simeq l$. Note that since the field extension $l/k$ is of degree $2$, we have $l\simeq k(\sqrt{\lambda})$ for some $\lambda \in k^\times/(k^\times)^2$. Therefore, we conclude that $HH_0(A)\simeq k(\sqrt{\lambda})$. Under the identification $\mathrm{End}_{\NChow(k)}(U(A))\simeq \bbZ[G/H]$, the additive functor $HH_0(-)$ sends the element $\sigma$ to the conjugation automorphism $\sqrt{\lambda} \mapsto - \sqrt{\lambda}$ of the $k$-vector space $k(\sqrt{\lambda})$. Consequently, making use of the identification $\mathrm{End}_{\mathrm{vect}(k)}(HH_0(A))\simeq M_{2\times 2}(k)$ induced by the basis $\{1, \sqrt{\lambda}\}$ of $k(\sqrt{\lambda})$, we conclude that the functor $HH_0(-)$ induces the following ring homomorphism:
\begin{eqnarray*}
\bbZ[G/H] \too M_{2\times 2}(k) && \alpha^+1 + \alpha^-\sigma \mapsto \begin{bmatrix} \alpha^+ + \alpha^- & 0 \\ 0 & \alpha^+ - \alpha^- \end{bmatrix}_{2\times 2}\,.
\end{eqnarray*} 
This implies that the composition of $HH_0(\rho)\colon k(\sqrt{\lambda})^{\oplus n} \to HH_0(N\!\!M)$ with $HH_0(\varrho)\colon HH_0(N\!\!M) \to k(\sqrt{\lambda})^{\oplus n}$, in the category $\mathrm{vect}(k)$, admits the following block matrix representation: 

\begin{equation}\label{eq:matrix}
HH_0(\varrho \circ \rho) = \left[\begin{array}{cc|cc|cc}
\alpha^+_{11} + \alpha^-_{11} & 0& \cdots &  \cdots& \alpha^+_{1n} + \alpha^-_{1n}& 0\\
0 & \alpha^+_{11} - \alpha^-_{11} & \cdots & \cdots&0  & \alpha^+_{1n} - \alpha^-_{1n} \\
\hline
\vdots& \vdots & &&\vdots&\vdots\\
\vdots & \vdots & &&\vdots&\vdots\\
\hline
\alpha^+_{n1} + \alpha^-_{n1}& 0 & \cdots &\cdots&\alpha^+_{nn} + \alpha^-_{nn}&0\\
0& \alpha^+_{n1} - \alpha^-_{n1} & \cdots &\cdots&0 & \alpha^+_{nn} - \alpha^-_{nn}
\end{array}
\right]_{2n\times 2n}\,.
\end{equation}
Note that \eqref{eq:relations} implies that all the elements in the diagonal of the matrix \eqref{eq:matrix} are odd and that all the elements outside the diagonal are even. Therefore, a simple inductive argument shows that the determinant of \eqref{eq:matrix} is odd. In particular, the determinant of \eqref{eq:matrix} is non-zero. This implies that the homomorphism $HH_0(\rho)$ is injective and consequently that $\mathrm{dim}_k HH_0(N\!\!M)\geq 2n$. Since the integer $n \geq 1$ is arbitrary and the dimension of the $k$-vector space $HH_0(N\!\!M)$ is finite, we hence arrive to a contradiction. Therefore, there exists an integer $m\geq 0$ such that $N\!\!M_{\bbF_2}$ contains $U(A)^{\oplus m}_{\bbF_2}$, but not $U(A)^{\oplus (m+1)}_{\bbF_2}$, as a direct summand.
\end{proof}

We now have all the ingredients necessary for the proof of Theorem \ref{thm:cancellation}.
\subsection*{Proof of Theorem \ref{thm:cancellation}}
Note first that we have induced isomorphisms:
\begin{eqnarray}
& N\!\!M_\bbQ\oplus U(B)^{\oplus n}_{\bbQ}\oplus U(A)_\bbQ \simeq N\!\!M_\bbQ\oplus U(B')^{\oplus n}_\bbQ \oplus U(A')_\bbQ\,\,\mathrm{in}\,\,\NNum(k)_\bbQ& \label{eq:induced-new} \\
& N\!\!M_{\bbF_2}\oplus U(B)_{\bbF_2}^{\oplus n}\oplus U(A)_{\bbF_2} \simeq N\!\!M_{\bbF_2}\oplus U(B')_{\bbF_2}^{\oplus n}\oplus U(A')_{\bbF_2}\,\,\mathrm{in}\,\,\mathrm{NRad}(k)_{\bbF_2}\,. & \label{eq:induced-new11}
\end{eqnarray}
Moreover, as explained in \cite[Thm.~2.1]{Azumaya}, we have isomorphisms $U(B)_\bbQ \simeq U(k)_\bbQ$ and $U(A)_\bbQ \simeq U(l)_\bbQ$ in the category $\NChow(k)_\bbQ$; similarly for $B'$ and $A'$. Consequently, the above isomorphism \eqref{eq:induced-new} reduces to $N\!\!M_\bbQ\oplus U(k)^{\oplus n}_\bbQ \oplus U(l)_\bbQ \simeq N\!\!M_\bbQ\oplus U(k)^{\oplus n}_\bbQ \oplus U(l')_\bbQ$. Since the category $\NNum(k)_\bbQ$ is abelian semi-simple (consult \cite[Thm.~4.27]{book}), we hence conclude that $U(l)_\bbQ \simeq U(l')_\bbQ$ in $\NNum(k)_\bbQ$. Thanks to Lemma \ref{lem:key1}, this implies that the $k$-algebras $l$ and $l'$ are isomorphic. Therefore, in the remainder of the proof we can (and will) assume without loss of generality that $l= l'$.

Let $H\subseteq G$ be the subgroup of index $2$ such that $\overline{k}^H=l$. As explained in the proof of Lemma \ref{lem:key2}, the $\bbF_2$-algebras of endomorphisms $\mathrm{End}_{\NChow(k)_{\bbF_2}}(U(A)_{\bbF_2})$ and $\mathrm{End}_{\NChow(k)_{\bbF_2}}(U(A')_{\bbF_2})$ can be identified with the group $\bbF_2$-algebra $\bbF_2[G/H]$. Consequently, by definition of the category of noncommutative radical motives, we obtain induced identifications:
\begin{eqnarray}\label{eq:computation1}
\mathrm{End}_{\mathrm{NRad}(k)_{\bbF_2}}(U(A)_{\bbF_2})\simeq \bbF_2 && \mathrm{End}_{\mathrm{NRad}(k)_{\bbF_2}}(U(A')_{\bbF_2})\simeq \bbF_2\,.
\end{eqnarray}

Let us assume by absurd that $U(A)\not\simeq U(A')$ in $\NChow(k)$. Thanks to Lemma \ref{lem:key2}, this implies that $[A]\neq [A']$ in $\mathrm{Br}(l)$ or, equivalently, that $\mathrm{ind}(A^\op \otimes_l A')\neq 1$. Since $\mathrm{ind}(A^\op \otimes_l A')$ divides the product $\mathrm{ind}(A^\op) \mathrm{ind}(A)$ and $\mathrm{ind}(A^\op)=\mathrm{ind}(A)$ and $\mathrm{ind}(A')$ are powers of $2$, the index $\mathrm{ind}(A^\op \otimes_l A')$ is also a power of $2$. Therefore, following the proof of Lemma \ref{lem:key2}, we observe that the composition map
$$ \Hom_{\NChow(k)_{\bbF_2}}(U(A)_{\bbF_2}, U(A')_{\bbF_2}) \times \Hom_{\NChow(k)_{\bbF_2}}(U(A')_{\bbF_2}, U(A)_{\bbF_2}) \too \Hom_{\NChow(k)_{\bbF_2}}(U(A)_{\bbF_2}, U(A)_{\bbF_2})$$
is equal to zero; similarly with $A$ and $A'$ interchanged. By definition of the category of noncommutative radical motives, this hence implies that 
\begin{equation}\label{eq:computation2}
\Hom_{\mathrm{NRad}(k)_{\bbF_2}}(U(A)_{\bbF_2}, U(A')_{\bbF_2}) = \Hom_{\mathrm{NRad}(k)_{\bbF_2}}(U(A')_{\bbF_2}, U(A)_{\bbF_2}) =0 \,.
\end{equation}
Thanks to Lemma \ref{lem:key}, we have isomorphisms $N\!\!M_{\bbF_2}\simeq M\!\!N \oplus U(A)^{\oplus m}_{\bbF_2}$ and $N\!\!M_{\bbF_2}\simeq M\!\!N' \oplus U(A')^{\oplus m'}_{\bbF_2}$ in the category $\mathrm{NRad}(k)_{\bbF_2}$ for some integers $m, m'\geq 0$, where $M\!\!N$, resp. $M\!\!N'$, is a noncommutative radical motive which does not contains $U(A)_{\bbF_2}$, resp. $U(A')_{\bbF_2}$, as a direct summand. Note that, thanks to \eqref{eq:computation2}, these direct sum decompositions also yield a direct sum decomposition $N\!\!M_{\bbF_2}\simeq M\!\!N'' \oplus U(A)^{\oplus m}_{\bbF_2} \oplus U(A')^{\oplus m'}_{\bbF_2}$ in $\mathrm{NRad}(k)_{\bbF_2}$, where $M\!\!N''$ is a noncommutative radical motive which does not contains $U(A)_{\bbF_2}$ neither $U(A')_{\bbF_2}$ as a direct summand. Consequently, the above isomorphism \eqref{eq:induced-new11} can be re-written as follows:
\begin{equation}\label{eq:induced2}
M\!\!N''\oplus U(B)_{\bbF_2}^{\oplus n}\oplus U(A)^{\oplus(m+1)}_{\bbF_2} \oplus U(A')^{\oplus m'}_{\bbF_2}\simeq M\!\!N''\oplus U(B')_{\bbF_2}^{\oplus n}\oplus U(A)_{\bbF_2}^{\oplus m} \oplus U(A')^{\oplus(m'+1) }_{\bbF_2}\,.
\end{equation}
Now, by combining the above computations \eqref{eq:computation1}-\eqref{eq:computation2} with Lemma \ref{lem:computation-B-A} and with the fact that $M\!\!N''$ does not contains $U(A)_{\bbF_2}$ neither $U(A')_{\bbF_2}$ as a direct summand, we observe that the above isomorphism \eqref{eq:induced2} restricts to an isomorphism $U(A)^{\oplus (m+1)}_{\bbF_2} \simeq U(A)^{\oplus m}_{\bbF_2}$. By applying the functor $\Hom_{\mathrm{NRad}(k)_{\bbF_2}}(U(A)_{\bbF_2}, -)$ to this latter isomorphism, we hence obtain an induced isomorphism $\bbF_2^{\oplus (r+1)}\simeq \bbF_2^{\oplus r}$ of $\bbF_2$-vector spaces, which is a contradiction. Therefore, we conclude that $U(A)\simeq U(A')$ in $\NChow(k)$.
\begin{remark}[Motivation for noncommutative radical motives]
The category of noncommutative numerical motives has several good properties. In particular, it is abelian semi-simple. However, given a field extension $l/k$ of degree $2$ and a central simple $l$-algebra $A$, we observe that $U(A)_{\mathbb{F}_2}$ is isomorphic to the zero object in the category $\NNum(k)_{\mathbb{F}_2}$. Consequently, the proof of Theorem \ref{thm:cancellation} fails miserably when the category $\mathrm{NRad}(k)_{\mathbb{F}_2}$ is replaced by the category $\NNum(k)_{\mathbb{F}_2}$. This was our main motivation for the use of the category of noncommutative radical motives (instead of the category of noncommutative numerical motives), which ``interpolates'' between noncommutative Chow motives and noncommutative numerical motives.
\end{remark}
\section{Proof of Theorem \ref{thm:main-1}}
\noindent
{\bf Proof of item (i).} Let $X$ and $Y$ be two varieties. As proved in \cite[\S2 Cor.~3.5.12 b)]{Integration}, if $[X]=[Y]$ in the Grothendieck ring of varieties $K_0\mathrm{Var}(k)$, then $\mathrm{dim}(X)=\mathrm{dim}(Y)$. Consequently, if $[\mathrm{Iv}(A,\ast)]=[\mathrm{Iv}(A',\ast')]$, we conclude that $\mathrm{dim}(\mathrm{Iv}(A,\ast))=\mathrm{dim}(\mathrm{Iv}(A',\ast'))$. Using the fact that $\mathrm{dim}(\mathrm{Iv}(A,\ast))=\mathrm{deg}(A)-2$ and $\mathrm{dim}(\mathrm{Iv}(A',\ast'))=\mathrm{deg}(A')-2$, this implies that $\mathrm{deg}(A)=\mathrm{deg}(A')$.

\medskip

\noindent
{\bf Proof of item (ii).} Following \cite[Thm~2.1 and Rk.~2.3]{Homogeneous}, we have the following computation
\begin{eqnarray*}
U(\mathrm{Iv}(A,\ast)) \simeq \begin{cases} 
U(k)^{\oplus (\mathrm{deg}(A)-2)} \oplus U(C_0(A,\ast)) & \mathrm{deg}(A)\,\,\mathrm{odd} \\
U(k)^{\oplus (\mathrm{deg}(A)/2-1)} \oplus U(A)^{\oplus(\mathrm{deg}(A)/2-1)} \oplus U(C_0(A,\ast)) & \mathrm{deg}(A)\,\,\mathrm{even}
\end{cases}
\end{eqnarray*}
in the category $\NChow(k)$; similarly with $(A,\ast)$ replaced by $(A',\ast')$. If $[\mathrm{Iv}(A,\ast)]=[\mathrm{Iv}(A',\ast')]$ in $K_0\mathrm{Var}(k)$, then $\mu_{\mathrm{nc}}(\mathrm{Iv}(A,\ast))=\mu_{\mathrm{nc}}(\mathrm{Iv}(A',\ast'))$ in $K_0(\mathrm{NChow}(k))$, where $\mu_{\mathrm{nc}}$ is the motivic measure of Proposition \ref{prop:measure}. Thanks to the definition of the Grothendieck ring $K_0(\mathrm{NChow}(k))$, this implies that there exists a noncommutative Chow motive $N\!\!M \in \NChow(k)$ such that
\begin{eqnarray}\label{eq:star-last-last}
&&\quad \begin{cases} N\!\!M\oplus U(k)^{\oplus d}\oplus U(C_0(A,\ast))\simeq N\!\!M\oplus U(k)^{\oplus d}\oplus U(C_0(A',\ast'))& \mathrm{deg}(A)\,\,\mathrm{odd}\\
N\!\!M\oplus U(k)^{\oplus d} \oplus U(A)^{\oplus d} \oplus U(C_0(A,\ast))\simeq N\!\!M\oplus U(k)^{\oplus d} \oplus U(A')^{\oplus d} \oplus U(C_0(A',\ast')) & \mathrm{deg}(A)\,\,\mathrm{even}\,,
\end{cases} 
\end{eqnarray}
where $d:=\mathrm{deg}(A)-2$ when $\mathrm{deg}(A)$ is odd and $d:=\mathrm{deg}(A)/2-1$ when $\mathrm{deg}(A)$ is even. Therefore, the proof follows now from Lemma \ref{lem:equality} below.
\begin{lemma}\label{lem:equality}
Let $(A,\ast)$ and $(A',\ast')$ be two central simple $k$-algebras with involutions of orthogonal type such that $\mathrm{deg}(A)$ and $\mathrm{deg}(A')$ are even. Given any noncommutative Chow motive $N\!\!M \in \NChow(k)$ and integer $n \geq 0$, we have the following implication:
\begin{eqnarray*}
N\!\!M \oplus U(A)^{\oplus n} \oplus U(C_0(A,\ast)) \simeq N\!\!M \oplus U(A')^{\oplus n} \oplus U(C_0(A',\ast')) & \Rightarrow & \big(\delta(A,\ast) \in (k^\times)^2 \Leftrightarrow \delta(A',\ast') \in (k^\times)^2 \big)\,.
\end{eqnarray*}
\end{lemma}
\begin{proof}
Note first that we have an induced isomorphism
\begin{equation}\label{eq:isomorphism-final11}
N\!\!M_\bbQ \oplus U(A)^{\oplus n}_\bbQ \oplus U(C_0(A,\ast))_\bbQ\simeq N\!\!M_\bbQ \oplus U(A')^{\oplus n}_\bbQ  \oplus U(C_0(A',\ast'))_\bbQ\,\,\mathrm{in}\,\,\NNum(k)_\bbQ\,.
\end{equation}
Moreover, as proved in \cite[Thm.~2.1]{Azumaya}, we have isomorphisms $U(A)_\bbQ \simeq U(k)_\bbQ$ and $U(A')_\bbQ \simeq U(k)_\bbQ$ in $\NChow(k)_\bbQ$. Consequently, since the category $\NNum(k)_\bbQ$ is abelian semi-simple (consult \cite[Thm.~4.27]{book}), we conclude from \eqref{eq:isomorphism-final11} that $U(C_0(A,\ast))_\bbQ \simeq U(C_0(A',\ast'))_\bbQ$ in $\NNum(k)_\bbQ$. 

Let us now prove the implication $\delta(A,\ast) \in (k^\times)^2 \Rightarrow \delta(A',\ast') \in (k^\times)^2$; we will assume that $\delta(A,\ast) \in (k^\times)^2$ and, by absurd, that $\delta(A',\ast') \not\in (k^\times)^2$. These assumptions imply that $C_0(A,\ast)\simeq C^+_0(A,\ast) \times C^-_0(A,\ast)$ is a product of two central simple $k$-algebras and that $C_0(A',\ast')$ is a central simple algebra over its center $l':=k(\sqrt{\delta(A',\ast')})$. Hence, we obtain an isomorphism $U(C_0(A,\ast))\simeq U(C^+_0(A,\ast))\oplus U(C^-_0(A,\ast))$ in the category $\NChow(k)$. As proved in \cite[Thm.~2.1]{Azumaya}, we have moreover isomorphisms $U(C^+_0(A,\ast))_\bbQ\simeq U(k)_\bbQ$, $U(C^-_0(A,\ast))_\bbQ\simeq U(k)_\bbQ$ and $U(C_0(A',\ast'))_\bbQ\simeq U(l')_\bbQ$ in the category $\NChow(k)_\bbQ$. Therefore, we obtain from the above considerations an isomorphism $U(k)_\bbQ \oplus U(k)_\bbQ \simeq U(l')_\bbQ$ in $\NNum(k)_\bbQ$ and, consequently, an induced isomorphism of $\bbQ$-vector spaces:
\begin{equation}\label{eq:equality-key}
\mathrm{Hom}_{\NNum(k)_\bbQ}(U(k)_\bbQ, U(k)_\bbQ \oplus U(k)_\bbQ) \simeq \mathrm{Hom}_{\NNum(k)_\bbQ}(U(k)_\bbQ, U(l')_\bbQ)\,.
\end{equation}
On the one hand, the left-hand side of \eqref{eq:equality-key} is isomorphic to $\bbQ\oplus \bbQ$. On the other hand, since the field extension $l'/k$ is of degree $2$, the following composition map
$$ \Hom_{\NChow(k)_\bbQ}(U(k)_\bbQ, U(l')_\bbQ) \times \Hom_{\NChow(k)_\bbQ}(U(l')_\bbQ, U(k)_\bbQ) \too \Hom_{\NChow(k)_\bbQ}(U(k)_\bbQ, U(k)_\bbQ)$$
corresponds to the bilinear pairing $\bbQ \times \bbQ \to \bbQ, (\alpha,\beta) \mapsto 2\alpha\beta$. By definition of the category $\NNum(k)_\bbQ$, this implies that the right-hand side of \eqref{eq:equality-key} is isomorphic to $\bbQ$, which is a contradiction! The proof of the converse implication is similar; simply interchange $(A,\ast)$ and $(A',\ast')$.
\end{proof}

\medbreak

\noindent
{\bf Proof of item (iii).} We start with the following cancellation result:
\begin{proposition}\label{prop:aux22}
Let $(A,\ast)$ and $(A',\ast')$ be two central simple $k$-algebras with involutions of orthogonal type such that $\mathrm{deg}(A)$ and $\mathrm{deg}(A')$ are even. Given any noncommutative Chow motive $N\!\!M \in \NChow(k)$ and integer $n\geq 0$, we have the following implication:
\begin{eqnarray*}\label{eq:implication-final2}
N\!\!M \oplus U(A)^{\oplus n}\oplus U(C_0(A,\ast))\simeq N\!\!M \oplus U(A')^{\oplus n}\oplus U(C_0(A',\ast')) &\Rightarrow & U(C_0(A,\ast))\simeq U(C_0(A',\ast'))\,.
\end{eqnarray*}
In the particular case where $n=1$, $\mathrm{deg}(A)\equiv 0$ (mod $4$), $\mathrm{deg}(A')\equiv 0$ (mod $4$), and $\delta(A,\ast) \in (k^\times)^2$ (or, equivalently, $\delta(A',\ast') \in (k^\times)^2$), we assume moreover that $A$ and $A'$ are split.
\end{proposition}
\begin{proof}
We consider first the case where $\delta(A,\ast)\in (k^\times)^2$. Thanks to Lemma \ref{lem:equality}, we also have $\delta(A',\ast') \in (k^\times)^2$. This implies that $C_0(A,\ast)\simeq C^+_0(A,\ast) \times C^-_0(A,\ast)$ and $C_0(A',\ast')\simeq C^+_0(A',\ast') \times C^-_0(A',\ast')$ are products of central simple $k$-algebras. Hence, we obtain induced isomorphisms
\begin{eqnarray}\label{eq:square-box}
U(C_0(A,\ast))\simeq U(C_0^+(A,\ast))\oplus U(C_0^-(A,\ast)) && U(C_0(A',\ast'))\simeq U(C_0^+(A',\ast'))\oplus U(C_0^-(A',\ast'))\,,
\end{eqnarray}
which lead to the following isomorphism of noncommutative Chow motives:
$$
N\!\!M \oplus U(A)^{\oplus n} \oplus U(C^+_0(A,\ast)) \oplus U(C^-_0(A,\ast)) \simeq N\!\!M \oplus U(A')^{\oplus n} \oplus U(C^+_0(A',\ast')) \oplus U(C^-_0(A',\ast'))\,.
$$
Note that thanks to Proposition \ref{prop:cancellation-0}, the latter isomorphism further restricts to an isomorphism 
\begin{equation}\label{eq:quasi-final11}
U(A)^{\oplus n}\oplus U(C^+_0(A,\ast))\oplus U(C^-_0(A,\ast)) \simeq U(A')^{\oplus n}\oplus U(C^+_0(A',\ast'))\oplus U(C^-_0(A',\ast'))\,\,\mathrm{in}\,\,\NChow(k)\,.
\end{equation}
When $\mathrm{deg}(A)\equiv 2$ (mod $4$), we have the following relations in the Brauer group:
\begin{eqnarray}\label{eq:relations1}
2[C^+_0(A,\ast)]=[A] & 3[C^+_0(A,\ast)]=[C^-_0(A,\ast)] & 4[C^+_0(A,\ast)]=[k]\,. 
\end{eqnarray}
In the same vein, when $\mathrm{deg}(A)\equiv 0$ (mod $4$), we have the following relations in the Brauer group:
\begin{eqnarray}\label{eq:relations2}
2[C^+_0(A,\ast)]=[k] & 2[C^+_0(A,\ast)]=[k] & [C^+_0(A,\ast)]+ [C^-_0(A,\ast)]=[A]\,. 
\end{eqnarray}
This implies, in particular, that the Brauer classes $[C^+_0(A,\ast)]$, $[C^-_0(A,\ast)]$, and $[A]$, belong to $\mathrm{Br}(k)\{2\}$; similarly with $A$ and $\ast$ replaced by $A'$ and $\ast'$. Consequently, by applying Theorem \ref{thm:Brauer1} to the above isomorphism \eqref{eq:quasi-final11}, we conclude that the following two sets of Brauer classes are the same up to permutation:
\begin{eqnarray*}
\{\underbrace{[A],\ldots, [A]}_{n\text{-}\mathrm{copies}}, [C_0^+(A,\ast)], [C_0^-(A,\ast)]\} && \{\underbrace{[A'],\ldots, [A']}_{n\text{-}\mathrm{copies}}, [C_0^+(A',\ast')], [C_0^-(A',\ast')]\}\,.
\end{eqnarray*}
When $n\neq 1$, the above relations \eqref{eq:relations1}-\eqref{eq:relations2} imply that
\begin{eqnarray}\label{eq:equalities}
\begin{cases} [C_0^+(A,\ast)]=[C_0^+(A',\ast')] \\ [C_0^-(A,\ast)]=[C_0^-(A',\ast')] \end{cases} & \mathrm{or} & \begin{cases} [C_0^+(A,\ast)]=[C_0^-(A',\ast')]\\  [C_0^-(A,\ast)]=[C_0^+(A',\ast')]\end{cases}\,.
\end{eqnarray}
Similarly, when $n=1$ and $\mathrm{deg}(A)\equiv 2$ (mod $4$) or $\mathrm{deg}(A')\equiv 2$ (mod $4$), the above relations \eqref{eq:relations1}-\eqref{eq:relations2} yield the equalities \eqref{eq:equalities}. The same happens in the particular case where $n=1$, $\mathrm{deg}(A)\equiv 0$ (mod $4$), $\mathrm{deg}(A')\equiv 0$ (mod $4$), and $A$ and $A'$ are split. Making use of Theorem \ref{thm:Brauer}, we hence conclude from the above isomorphisms \eqref{eq:square-box} that in all these different cases we have $U(C_0(A,\ast))\simeq U(C_0(A',\ast'))$ in $\NChow(k)$.

We now consider the case where $\delta(A,\ast)\notin (k^\times)^2$. Thanks to Lemma \ref{lem:equality}, we also have $\delta(A',\ast')\notin (k^\times)^2$. This implies that $C_0(A,\ast)$ and $C_0(A',\ast')$ are central simple algebras over their centers $l:=k(\sqrt{\delta(A,\ast)})$ and $l':=k(\sqrt{\delta(A',\ast')})$, respectively. Note that since the field extension $l/k$ is of degree $2$ and $\mathrm{dim}_k(C_0(A,\ast))=2^{\mathrm{deg}(A)-1}$, the index of the central simple $l$-algebra $C_0(A,\ast)$ is a power of $2$; similarly for the central simple $l'$-algebra $C_0(A',\ast')$. Consequently, the proof follows now from Theorem \ref{thm:cancellation}.
\end{proof}
\begin{corollary}\label{cor:aux2}
Under the assumptions of Proposition \ref{prop:aux22}, we have the following implication:
\begin{eqnarray}\label{eq:implication-final222}
N\!\!M \oplus U(A)^{\oplus n}\oplus U(C_0(A,\ast))\simeq N\!\!M \oplus U(A')^{\oplus n}\oplus U(C_0(A',\ast')) &\Rightarrow & U(A)\simeq U(A')\,.
\end{eqnarray}
\end{corollary}
\begin{proof}
Proposition \ref{prop:aux22} yields an isomorphism $U(C_0(A,\ast))\simeq U(C_0(A',\ast'))$. Therefore, by applying Proposition \ref{prop:cancellation-0} to the left-hand side of \eqref{eq:implication-final222}, we conclude that $U(A)^{\oplus n}\simeq U(A')^{\oplus n}$. Thanks to Theorems \ref{thm:Brauer} and \ref{thm:Brauer1}, this hence implies that $U(A)\simeq U(A')$ in $\NChow(k)$.
\end{proof}
Recall from the proof of item (ii) of Theorem \ref{thm:main-1} that if $[\mathrm{Iv}(A,\ast)]=[\mathrm{Iv}(A',\ast')]$, then we obtain the above computation \eqref{eq:star-last-last} in the category of noncommutative Chow motives. Consequently, by applying Proposition \ref{prop:cancellation-0} to the above isomorphism \eqref{eq:star-last-last} when $\mathrm{deg}(A)$ is odd, we conclude that $U(C_0(A,\ast))\simeq U(C_0(A',\ast'))$ in $\NChow(k)$. In the same vein, by applying Proposition \ref{prop:aux22} to the above isomorphism \eqref{eq:star-last-last} when $\mathrm{deg}(A)$ is even, we conclude that $U(C_0(A,\ast))\simeq U(C_0(A',\ast'))$ in $\NChow(k)$; note that Proposition \ref{prop:aux22} can be applied because in the particular case where $\mathrm{deg}(A)=4$ and $\delta(A,\ast)\in (k^\times)^2$ (or, equivalently, $\mathrm{deg}(A')=4$ and $\delta(A',\ast')\in (k^\times)^2$), we assume moreover that $(A,\ast)$ and $(A',\ast')$ satisfy condition $(\star)$. Therefore, since $\mathrm{deg}(A)=\mathrm{deg}(A')$, the proof follows now from Proposition \ref{prop:aux2} below.
\begin{proposition}\label{prop:aux2}
Let $(A,\ast)$ and $(A',\ast')$ be two central simple $k$-algebras with involutions of orthogonal type such that $\mathrm{deg}(A)=\mathrm{deg}(A')$. Under these assumptions, we have the following implication:
\begin{eqnarray*}
U(C_0(A,\ast)) \simeq U(C_0(A',\ast'))\,\,\mathrm{in}\,\,\NChow(k) & \Rightarrow & C_0(A,\ast) \simeq C_0(A',\ast')\,.
\end{eqnarray*} 
\end{proposition}
\begin{proof}
Recall that $\mathrm{dim}_k(C_0(A,\ast))=2^{\mathrm{deg}(A)-1}$ and $\mathrm{dim}_k(C_0(A',\ast'))=2^{\mathrm{deg}(A')-1}$. Therefore, since $\mathrm{deg}(A)=\mathrm{deg}(A')$, we have $\mathrm{dim}_k(C_0(A,\ast))=\mathrm{dim}_k(C_0(A',\ast'))$.

We consider first the case where $\mathrm{deg}(A)$ is odd. In this case, $C_0(A,\ast)$ and $C_0(A',\ast')$ are central simple $k$-algebras. Hence, Theorem \ref{thm:Brauer} implies that $[C_0(A,\ast)]=[C_0(A',\ast')]$. Consequently, by combining the equality $\mathrm{dim}_k(C_0(A,\ast)) = \mathrm{dim}_k(C_0(A',\ast'))$ with the Wedderburn theorem, we conclude that $C_0(A,\ast)\simeq C_0(A',\ast')$.

We now consider the case where $\mathrm{deg}(A)$ is even and $\delta(A,\ast)\in (k^\times)^2$. Thanks to Lemma \ref{lem:equality}, we also have $\delta(A',\ast') \in (k^\times)^2$. This implies that $C_0(A,\ast)\simeq C^+_0(A,\ast) \times C^-_0(A,\ast)$ and $C_0(A',\ast')\simeq C^+_0(A',\ast') \times C^-_0(A',\ast')$ are products of central simple $k$-algebras. Hence, we obtain induced isomorphisms
\begin{eqnarray*}
U(C_0(A,\ast)) \simeq U(C_0^+(A,\ast))\oplus U(C_0^-(A,\ast)) && U(C_0(A',\ast')) \simeq U(C_0^+(A',\ast'))\oplus U(C_0^-(A',\ast'))\,,
\end{eqnarray*}
which lead to the following isomorphism
\begin{equation}\label{eq:square}
U(C^+_0(A,\ast)) \oplus U(C^-_0(A,\ast)) \simeq U(C^+_0(A',\ast'))\oplus U(C^-_0(A',\ast'))\,\,\mathrm{in}\,\,\NChow(k)\,.
\end{equation}
As explained in the proof of Proposition \ref{prop:aux22}, the Brauer classes $[C_0^+(A,\ast)]$ and $[C_0^-(A,\ast)]$ belong to $\mathrm{Br}(k)\{2\}$; similarly with $(A, \ast)$ replaced by $(A', \ast')$. Therefore, by applying Theorem \ref{thm:Brauer1} to the above isomorphism \eqref{eq:square}, we conclude that the sets of Brauer classes $\{[C_0^+(A,\ast)], [C_0^-(A,\ast)]\}$ and $\{[C_0^+(A',\ast')], [C_0^-(A',\ast')]\}$ are the same up to permutation.
Thanks to the assumption $\mathrm{deg}(A)=\mathrm{deg}(A')$, to the following equalities
\begin{eqnarray*}
\mathrm{dim}_k(C^+_0(A,\ast))=\mathrm{dim}_k(C^-_0(A,\ast))=2^{\mathrm{deg}(A)-2} && \mathrm{dim}_k(C^+_0(A',\ast'))=\mathrm{dim}_k(C^-_0(A',\ast'))=2^{\mathrm{deg}(A')-2}\,,
\end{eqnarray*}
and to the Wedderburn theorem, this implies that
\begin{eqnarray*}
\begin{cases} C_0^+(A,\ast)\simeq C_0^+(A',\ast') \\ C_0^-(A,\ast)\simeq C_0^-(A',\ast') \end{cases} & \mathrm{or} & \begin{cases} C_0^+(A,\ast)\simeq C_0^-(A',\ast') \\ C_0^-(A,\ast)\simeq C_0^+(A',\ast') \end{cases}\,.
\end{eqnarray*}
In both cases, we have an isomorphism $C_0(A,\ast) \simeq C_0(A',\ast')$.

Finally, we consider the case where $\mathrm{deg}(A)$ is even and $\delta(A,\ast)\notin (k^\times)^2$. Thanks to Lemma \ref{lem:equality}, we also have $\delta(A',\ast') \notin (k^\times)^2$. This implies that $C_0(A,\ast)$ and $C_0(A',\ast)$ are central simple algebras over their centers $l:=k(\sqrt{\delta(A,\ast)})$ and $l':=k(\sqrt{\delta(A',\ast')})$, respectively. Hence, we obtain an induced isomorphism $U(C_0(A,\ast))_\bbQ\simeq U(C_0(A',\ast'))_\bbQ$ in $\NChow(k)_\bbQ$. As proved in \cite[Thm.~2.1]{Azumaya}, we have isomorphisms $U(C_0(A,\ast))_\bbQ\simeq U(l)_\bbQ$ and $U(C_0(A',\ast'))_\bbQ \simeq U(l')_\bbQ$ in $\NChow(k)_\bbQ$. Thanks to Lemma \ref{lem:key1}, this implies that $l\simeq l'$. Therefore, we can consider the central simple $l$-algebra $\mathbf{C}_0:=C_0(A',\ast') \otimes_{l'} l$. Note that since the $k$-algebras $C_0(A',\ast')$ and $\mathbf{C}_0$ are isomorphic, we have an isomorphism $U(C_0(A',\ast'))\simeq U(\mathbf{C}_0)$ in $\NChow(k)$. Thanks to Lemma \ref{lem:key2}, this latter isomorphism, combined with the hypothesis that $U(C_0(A,\ast))\simeq U(C_0(A',\ast'))$ in $\NChow(k)$, implies that $[C_0(A,\ast)]=[\mathbf{C}_0]$ in $\mathrm{Br}(l)$. Consequently, making use of the following equalities 
$$ \mathrm{dim}_l(C_0(A,\ast)) = \mathrm{dim}_k(C_0(A,\ast))/2= \mathrm{dim}_k(C_0(A',\ast'))/2=\mathrm{dim}_k(\mathbf{C}_0)/2=\mathrm{dim}_l(\mathbf{C}_0)$$ 
and of the Wedderburn theorem, we conclude that the $l$-algebras $C_0(A,\ast)$ and $\mathbf{C}_0$ are isomorphic. This implies that $C_0(A,\ast)\simeq C_0(A',\ast')$.
\end{proof}

\medbreak

\noindent
{\bf Proof of item (iv).} Thanks to item (i) of Theorem \ref{thm:main-1}, we have $\mathrm{deg}(A)=\mathrm{deg}(A')$. Moreover, recall from Example \ref{ex:2} that when $\mathrm{deg}(A)=\mathrm{deg}(A')$ is odd, both $A$ and $A'$ are isomorphic to $\mathrm{M}_{\mathrm{deg}(A)\times\mathrm{deg}(A)}(k)$. Hence, it suffices to consider the case where $\mathrm{deg}(A)=\mathrm{deg}(A')$ is even. Recall from item (ii) of Theorem \ref{thm:main-1} that if $[\mathrm{Iv}(A,\ast)]=[\mathrm{Iv}(A',\ast')]$, then we obtain the above computation \eqref{eq:star-last-last} in the category of noncommutative Chow motives. By applying Corollary \ref{cor:aux2} to the above isomorphism \eqref{eq:star-last-last}, we hence conclude that $U(A)\simeq U(A')$ in $\NChow(k)$; note that Corollary \ref{cor:aux2} can be applied because in the particular case where $\mathrm{deg}(A)=4$ and $\delta(A,\ast)\in (k^\times)^2$ (or, equivalently, $\mathrm{deg}(A')=4$ and $\delta(A',\ast')\in (k^\times)^2$), we assume moreover that $(A,\ast)$ and $(A',\ast')$ satisfy condition $(\star)$. Therefore, since $\mathrm{deg}(A)=\mathrm{deg}(A')$, it follows now from Theorem \ref{thm:Brauer} and from the Wedderburn theorem that $A\simeq A'$.

\medbreak

\noindent
{\bf Proof of item (v).} We start with a general result of independent interest:
\begin{proposition}
Let $q$ and $q'$ be two non-degenerate quadratic forms such that $\mathrm{dim}(q)=\mathrm{dim}(q')$. If $\mu_{\mathrm{c}}([Q_q])=\mu_{\mathrm{c}}([Q_{q'}])$ in $K_0(\Num(k))$, then $q$ and $q'$ are both isotropic or both anisotropic.
\end{proposition}
\begin{proof}
Let $l/k$ be a field extension making the quadratic form $q\otimes_kl$ isotropic. As explained in \cite[\S4.2.3]{Andre}, extension of scalars along the inclusion $k \subset l$ gives rise to the commutative diagram:
$$
\xymatrix{
\Chow(k) \ar[rr]^-{-\times_k l} \ar[d] && \Chow(l) \ar[d] \\
\Num(k) \ar[rr]_-{-\times_k l} && \Num(l)\,.
}
$$
This leads to the following induced commutative diagram of abelian groups
\begin{equation}\label{eq:big-square}
\xymatrix{
\mathrm{Hom}_{\Chow(k)}(\mathfrak{h}(\mathrm{pt})(-d), \mathfrak{h}(Q_q)) \ar@{->>}[d] \ar[rr]^-{\phi_q} && \Hom_{\Chow(l)}(\mathfrak{h}(\mathrm{pt})(-d), \mathfrak{h}(Q_q\times_k l)) \ar@{->>}[d] \\
\mathrm{Hom}_{\Num(k)}(\mathfrak{h}(\mathrm{pt})(-d), \mathfrak{h}(Q_q)) \ar[rr]_-{\varphi_q} && \Hom_{\Num(l)}(\mathfrak{h}(\mathrm{pt})(-d), \mathfrak{h}(Q_q\times_k l))\,,
}
\end{equation}
where $d:=\mathrm{dim}(Q_q)=\mathrm{dim}(q)-2$. Since the quadratic form $q\otimes_k l$ is isotropic, we have an orthogonal sum decomposition $q \otimes_k l = \underline{q} \perp \bbH$. Consequently, as proved by Rost in \cite[Prop.~2]{Rost}, the Chow motive $\mathfrak{h}(Q_q \times_k l) = \mathfrak{h}(Q_{q\otimes_k l})$ admits the direct sum decomposition $\mathfrak{h}(\mathrm{pt}) \oplus \mathfrak{h}(Q_{\underline{q}})(-1) \oplus \bbZ(-d)$ in $\Chow(k)$. Since $\mathrm{dim}(Q_{\underline{q}})=d-2$, this decomposition implies that the upper and lower right corners of the commutative diagram \eqref{eq:big-square} are both isomorphic to $\bbZ$ and that the (vertical) homomorphism between them is the identity. Note that by definition (consult \S\ref{sec:preliminaries}), the upper and lower left corners of the commutative diagram \eqref{eq:big-square} are isomorphic to $\cZ^d_{\sim \mathrm{rat}}(Q_q)$ and $\cZ^d_{\sim \mathrm{num}}(Q_q)$, respectively. Note also that these are the abelian groups of $0$-cycles on $Q_q$ up to rational equivalence and numerical equivalence, respectively. Following Karpenko \cite[\S2]{Karpenko}, under the above identifications, the homomorphism $\phi_q$ corresponds to the degree map $\mathrm{deg}\colon \cZ^d_{\sim \mathrm{rat}}(Q_q) \to \bbZ$. Consequently, since two $0$-cycles in $Q_q$ are numerically trivial if and only if they have the same degree, we conclude that $\mathrm{cok}(\varphi_q) \simeq \mathrm{cok}(\phi_q)$. All the above holds {\em mutatis mutandis} with $q$ replaced by $q'$.

Now, let $l/k$ be a field extension making both quadratic forms $q\otimes_kl$ and $q'\otimes_k l$ isotropic. By definition of the Grothendieck ring $K_0(\Num(k))$, if $\mu_{\mathrm{c}}([Q_q]) = \mu_{\mathrm{c}}([Q_{q'}])$, then there exists a numerical motive $M \in \Num(k)$ such that $M\oplus \mathfrak{h}(Q_q)\simeq M\oplus \mathfrak{h}(Q_{q'})$. Let us choose an isomorphism $\theta\colon M\oplus \mathfrak{h}(Q_q) \to M \oplus \mathfrak{h}(Q_{q'})$. Note that since the extension of scalars functor $-\times_k l \colon \Num(k) \to \Num(l)$ is additive, it leads to the following commutative diagram of abelian groups:
$$
\xymatrix{
\Hom_{\Num(k)}(\mathfrak{h}(\mathrm{pt})(-d), M\oplus \mathfrak{h}(Q_q)) \ar[d]_-{\theta_\ast}^-\simeq \ar[rr]^-{\varphi\oplus \varphi_{q}} && \Hom_{\Num(l)}(\mathfrak{h}(\mathrm{pt})(-d), (M\times_k l)\oplus \mathfrak{h}(Q_q \times_k l)) \ar[d]^-{(\theta\times_k l)_\ast}_-{\simeq} \\
\Hom_{\Num(k)}(\mathfrak{h}(\mathrm{pt})(-d), M\oplus \mathfrak{h}(Q_{q'})) \ar[rr]_-{\varphi\oplus \varphi_{q'}} && \Hom_{\Num(l)}(\mathfrak{h}(\mathrm{pt})(-d), (M\times_k l)\oplus \mathfrak{h}(Q_{q'} \times_k l))\,.
}
$$
This implies that $\mathrm{cok}(\varphi\oplus \varphi_q) \simeq \mathrm{cok}(\varphi\oplus \varphi_{q'})$ or, equivalently, that $\mathrm{cok}(\varphi) \oplus \mathrm{cok}(\varphi_q) \simeq \mathrm{cok}(\varphi) \oplus \mathrm{cok}(\varphi_{q'})$. Using the fact that these abelian groups are finitely generated (consult \cite[Prop.~3.2.7.1]{Andre}), we hence conclude that $\mathrm{cok}(\varphi_q)\simeq \mathrm{cok}(\varphi_{q'})$. The proof follows now from a celebrated result of Springer (consult \cite[Cor.~71.3]{Book-quadrics}\cite{Springer}), which asserts that $\mathrm{cok}(\varphi_q)\simeq 0$ if $q$ is isotropic and $\mathrm{cok}(\varphi_q)\simeq \bbZ/2$ if $q$ is anisotropic.
\end{proof}
Let $P$ be a (fixed) ordering of $k$ and $k_P$ the associated real-closure of $k$. Note first that extension of scalars along the inclusion $k \subset k_P$ gives rise to a ring homomorphism $-\times_k k_P \colon K_0\mathrm{Var}(k) \to K_0\mathrm{Var}(k_P)$. Consequently, since by hypothesis $[\mathrm{Iv}(A,\ast)]=[\mathrm{Iv}(A',\ast')]$, we obtain the following equality in $K_0\mathrm{Var}(k_P)$:
\begin{equation}\label{eq:equalities-new}
[\mathrm{Iv}(A,\ast) \times_k k_P] = [\mathrm{Iv}(A',\ast') \times_k k_P] = [\mathrm{Iv}(A\otimes_k k_P, \ast \otimes_k k_P)]= [\mathrm{Iv}(A'\otimes_k k_P, \ast' \otimes_k k_P)]\,.
\end{equation}
Now, recall from item (iv) that if $[\mathrm{Iv}(A,\ast)]=[\mathrm{Iv}(A',\ast')]$, then $A\simeq A'$. This yields an induced isomorphism $A\otimes_k k_P \simeq A'\otimes_k k_P$. On one hand, if the central simple $k_P$-algebra $A\otimes_k k_P$ is not split, we have $\mathrm{sgn}_P(A,\ast)=0$; consult \cite[Thm.~1]{Tignol}. If this is the case, then $A'\otimes_k k_P$ is also not split and $\mathrm{sgn}_P(A',\ast')=0$. On the other hand, if $A\otimes_k k_P$ is split, we have an isomorphism of central simple $k_P$-algebras with involutions of orthogonal type $(A\otimes_k k_P, \ast \otimes_k k_P)\simeq (\mathrm{M}_{\mathrm{deg}(A)\times \mathrm{deg}(A)}(k_P), \ast_q)$ (consult Example \ref{ex:1}) and $\mathrm{sgn}_P(A,\ast) = |\mathrm{sgn}(q)|$. If this is the case, then $A'\otimes_k k_P$ is also split, we also have an isomorphism of central simple $k_P$-algebras with involution of orthogonal type $(A'\otimes_k k_P, \ast' \otimes_k k_P)\simeq (\mathrm{M}_{\mathrm{deg}(A)\times \mathrm{deg}(A)}(k_P), \ast_{q'})$, and $\mathrm{sgn}_P(A',\ast') = |\mathrm{sgn}(q')|$. Consequently, thanks to the above equality \eqref{eq:equalities-new} and to the fact that $\mathrm{Iv}(A\otimes_k k_P, \ast\otimes_k k_P)\simeq Q_q$ and $\mathrm{Iv}(A'\otimes_k k_P, \ast'\otimes_k k_P)\simeq Q_{q'}$, the proof of item (v) of Theorem \ref{thm:main-1} follows now from Proposition \ref{prop:key} below.
\begin{proposition}\label{prop:key}
Given two non-degenerate quadratic forms $q$ and $q'$ over a real-closed field $k$, we have the following implication:
\begin{eqnarray*}
[Q_q]=[Q_{q'}] \,\,\mathrm{in}\,\,K_0\mathrm{Var}(k) & \Rightarrow & |\mathrm{sgn}(q)|=|\mathrm{sgn}(q')|\,.
\end{eqnarray*}
\end{proposition}
\begin{proof}
Recall first from item (i) of Theorem \ref{thm:main} that if by hypothesis $[Q_q]=[Q_{q'}]$, then $\mathrm{dim}(q)=\mathrm{dim}(q')$. It is well-known that every real-closed field is, in particular, euclidean. Consequently, thanks to Sylvester's law of inertia (consult \cite[Prop.~31.5]{Book-quadrics}), the quadratic forms $q$ and $q'$ are similar to the following ones
\begin{eqnarray}
q= \langle \underbrace{1, \ldots, 1}_{m\text{-}\text{copies}}, \underbrace{-1, \ldots, -1}_{n\text{-}\text{copies}}\rangle && q'=\langle \underbrace{1, \ldots, 1}_{m'\text{-}\text{copies}}, \underbrace{-1, \ldots, -1}_{n'\text{-}\text{copies}} \rangle
\end{eqnarray}
for uniquely determined integers $m\geq n \geq 0$ and $m'\geq n'\geq 0$, respectively. Note that since $\mathrm{dim}(q)=\mathrm{dim}(q')$, we have $m+n=m'+n'$. In what follows we will assume that $|\mathrm{sgn}(q)|\neq |\mathrm{sgn}(q')|$ and, by absurd, that $[Q_q]=[Q_{q'}]$. If $|\mathrm{sgn}(q)|=m-n \neq m'-n' = |\mathrm{sgn}(q')|$, then $m\neq m'$. Hence, we can (and will) assume without loss of generality that $m > m'$.

Let us consider first the particular case where $n=0$. In this case, the quadratic form $q$ is anisotropic while the quadratic form $q'$ is isotropic (because $n'>0$). However, if $[Q_q]=[Q_{q'}]$, then it follows from Proposition \ref{prop:key} that $q$ and $q'$ are both isotropic or both anisotropic. Hence, we arrive to a contradiction.

Let us consider now the case where $n>0$. In this case, we have orthogonal sum decompositions
\begin{eqnarray*}
q = \underline{q} \perp n\bbH \,\,\,\, \text{with}\,\,\,\, \underline{q}= \langle \underbrace{1, \ldots \ldots, 1}_{(m-n)\text{-}\text{copies}}\rangle &&
q' = \underline{q'} \perp n\bbH \,\,\,\, \text{with} \,\,\,\, \underline{q'}= \langle \underbrace{1, \ldots\ldots, 1}_{(m'-n)\text{-}\text{copies}}, \underbrace{-1, \ldots\ldots, -1}_{(n'-n)\text{-}\text{copies}}\rangle \,,
\end{eqnarray*}
where $\bbH=\langle 1,-1\rangle$ stands for the hyperbolic plane. Note that the quadratic form $\underline{q}$ is anisotropic while the quadratic form $\underline{q'}$ is isotropic (because $n'-n >0$). Making use of Lemma \ref{lem:decomposition} below, we hence obtain the following computations in the Grothendieck ring of varieties:
$$ [Q_q]=(1+ \bbL + \cdots + \bbL^{n-1}) + [Q_{\underline{q}}]\cdot \bbL^n + \bbL^{\mathrm{dim}(q)-2} \cdot (1+ \bbL + \cdots + \bbL^{n-1})$$
$$ [Q_{q'}]=(1+ \bbL + \cdots + \bbL^{n-1}) + [Q_{\underline{q'}}]\cdot \bbL^n + \bbL^{\mathrm{dim}(q)-2} \cdot (1+ \bbL + \cdots + \bbL^{n-1})\,.$$
This implies that $[Q_q]\cdot \bbL^n = [Q_{q'}]\cdot \bbL^n$ in $K_0\mathrm{Var}(k)$. Since $\mu_{\mathrm{c}}(\bbL)=[\bbZ(-1)]$ (consult \cite[\S13.2.2]{Andre}) and the Lefschetz motive $\bbZ(-1)\in \Num(k)$ is $\otimes$-invertible, this further implies that 
$$\mu_{\mathrm{c}}([Q_{\underline{q}}]\cdot \bbL^n) = \mu_{\mathrm{c}}([Q_{\underline{q'}}]\cdot \bbL^n) \Leftrightarrow \mu_{\mathrm{c}}([Q_{\underline{q}}])\cdot [\bbZ(-n)] = \mu_{\mathrm{c}}([Q_{\underline{q'}}])\cdot [\bbZ(-n)] \Leftrightarrow \mu_{\mathrm{c}}([Q_{\underline{q}}]) = \mu_{\mathrm{c}}([Q_{\underline{q'}}])\,.$$
Consequently, Proposition \ref{prop:key} implies that $\underline{q}$ and $\underline{q'}$ are both isotropic or both anisotropic, which is a contradiction. This concludes the proof.
\end{proof}
\begin{lemma}\label{lem:decomposition}
Given a non-degenerate quadratic form $q$ and an integer $n\geq 1$, we have the following computation in $K_0\mathrm{Var}(k)$ ($\bbL$ stands for the Grothendieck class $[\bbA^1]$):
\begin{equation}\label{eq:equality-key-L}
[Q_{q\perp n \bbH}] =(1+ \bbL + \cdots + \bbL^{n-1}) + [Q_q]\cdot \bbL^n + \bbL^{\mathrm{dim}(q)}\cdot (\bbL^{n-1} + \cdots + \bbL^{2(n-1)})\,.
\end{equation}
\end{lemma}
\begin{proof}
Let $\cP(n)$ be the equality \eqref{eq:equality-key-L}. The proof will be done by induction on $n$. We start by proving the equality $\mathcal{P}(1)$ (for every quadratic form $q$). Note that there exists a choice of coordinates $x_1, \ldots, x_{\mathrm{dim}(q)}, u, w$ such that $(q \perp \bbH)(x_1, \ldots, x_{\mathrm{dim}(q)}, u, w)= q(x_1, \ldots, x_{\mathrm{dim}(q)}) + uw$. Following Rost \cite[Prop.~2]{Rost}, consider the stratification $\mathrm{pt} \subset Z \subset Q_{q \perp \bbH}$ of $Q_{q \perp \bbH}$, where $Z:=\{[x_1: \cdots : x_{\mathrm{dim}(q)}:0:w]\,|\, q(x_1, \ldots, x_{\mathrm{dim}(q)}) + 0w=0\}$ and $\mathrm{pt}:=[0: \cdots 0:1:0]$. By definition of $K_0\mathrm{Var}(k)$, we hence obtain the following computation:
\begin{equation}\label{eq:computation-aux}
[Q_{q \perp \bbH}]= [Z] + [Q_{q \perp \bbH} \backslash Z] = [\mathrm{pt}] + [Z\backslash \mathrm{pt}] + [Q_{q \perp \bbH} \backslash Z]\,.
\end{equation}
On the one hand, note that $Q_{q \perp \bbH} \backslash Z\simeq \bbA^{\mathrm{dim}(q)}$. This implies that $[Q_{q \perp \bbH} \backslash Z]= \bbL^{\mathrm{dim}(q)}$. On the other hand, as explained in \cite[Prop.~2]{Rost}, the following projection map 
\begin{eqnarray*}
Z\backslash \mathrm{pt} \too Q_q && [x_1:\cdots : x_{\mathrm{dim}(q)}: 0: w] \mapsto [x_1:\cdots : x_{\mathrm{dim}(q)}]
\end{eqnarray*}
is a $1$-dimensional vector bundle. Following \cite[\S2 Prop.~2.3.3]{Integration}, this implies that $[Z\backslash \mathrm{pt}]=[Q_q]\cdot [\bbA^1]$. Consequently, since $[\mathrm{pt}]=1$, we conclude that \eqref{eq:computation-aux} reduces to the following computation:
\begin{equation}\label{eq:computation-key}
[Q_{q\perp \bbH}] = 1 + [Q_q] \cdot \bbL + \bbL^{\mathrm{dim}(q)}\,.
\end{equation}
Let us now prove the implication $\mathcal{P}(n) \Rightarrow \mathcal{P}(n+1)$. If the equality $\mathcal{P}(n)$ holds (for every quadratic form $q$), then we obtain the following computation in $K_0\mathrm{Var}(k)$:
$$ [Q_{(q\perp \bbH)\perp n\bbH}] =(1+ \bbL + \cdots + \bbL^{n-1}) + [Q_{q\perp \bbH}]\cdot \bbL^n + \bbL^{\mathrm{dim}(q)+2} \cdot (\bbL^{n-1} + \cdots + \bbL^{2(n-1)})\,.$$
By combining it with \eqref{eq:computation-key}, we hence obtain the searched equality $\mathcal{P}(n+1)$.
\end{proof}
\section{Proof of Theorem \ref{thm:applications-1}}
\noindent
{\bf Proof of item (i).} Recall first from item (i) of Theorem \ref{thm:main-1} that if $[\mathrm{Iv}(A,\ast)]=[\mathrm{Iv}(A',\ast')]$, then $\mathrm{deg}(A)=\mathrm{deg}(A')$. Also, recall from Example \ref{ex:2} that the odd dimensional involution varieties are the odd dimensional quadrics. As proved in \cite[\S15.A]{Book-inv}, the assignment $q \mapsto C_0(q)$ gives rise to an bijection between the similarity classes of non-degenerate quadratic forms of dimension $3$ and the isomorphism classes of quaternion algebras. In the same vein, as proved in \cite[\S15.B]{Book-inv}, the assignment $(A,\ast) \mapsto C_0(A,\ast)$ gives rise to a bijection between the isomorphism classes of central simple $k$-algebras with involution of orthogonal type of degree $4$ and the isomorphism classes of quaternion algebras over an \'etale quadratic $k$-algebra. Consequently, given two central simple $k$-algebras with involutions of orthogonal type $(A,\ast)$ and $(A',\ast')$ such that $[\mathrm{Iv}(A,\ast)]=[\mathrm{Iv}(A',\ast')]$ and $\mathrm{deg}(A)=\mathrm{deg}(A')\leq 4$, the proof follows now from the combination of item (iii) of Theorem \ref{thm:main-1} with the well-known fact that $(A,\ast)$ and $(A',\ast')$ are isomorphic if and only if $\mathrm{Iv}(A,\ast)$ and $\mathrm{Iv}(A',\ast')$ are isomorphic. 

\medskip

\noindent
{\bf Proof of item (ii).} Given a central simple $k$-algebra $A$ and two involutions of orthogonal type $\ast$ and $\ast'$ on $A$, recall from \cite[Thm.~B]{LewisT} that if $I(k)^3$ is torsion-free, then we have the following equivalence:
\begin{equation}\label{eq:star}
(A,\ast) \simeq (A,\ast') \Leftrightarrow \big(C_0(A,\ast) \simeq C_0(A,\ast')\,\,\text{and}\,\,\mathrm{sgn}_P(A,\ast)=\mathrm{sgn}_P(A,\ast')\,\,\text{for}\,\,\text{every}\,\,\text{ordering}\,\,P\,\,\text{of}\,\,k\big)\,;
\end{equation}
in the case where $k$ is formally-real and $\mathrm{deg}(A')$ is even, we assume moreover that $\tilde{u}(k)< \infty$. Now, let $(A,\ast)$ and $(A',\ast')$ be two central simple $k$-algebras with involution of orthogonal type such that $[\mathrm{Iv}(A,\ast)]=[\mathrm{Iv}(A',\ast')]$. Thanks to item (iv) of Theorem \ref{thm:main-1}, we have an isomorphism $A\simeq A'$. Therefore, we can (and will) assume without loss of generality that $A= A'$. Consequently, the proof follows now from items (iii) and (v) of Theorem \ref{thm:main-1}, from the above equivalence \eqref{eq:star}, and from the well-known fact that $(A,\ast)$ and $(A',\ast')$ are isomorphic if and only if $\mathrm{Iv}(A,\ast)$ and $\mathrm{Iv}(A',\ast')$ are isomorphic.


\begin{thebibliography}{00}

\bibitem{Andre} Y.~Andr{\'e}, {\em Une introduction aux motifs (motifs purs, motifs mixtes, p\'eriodes)}. Panoramas et Synth\`eses, vol. {\bf 17}, Soci\'et\'e Math\'ematique de France, Paris, 2004.

\bibitem{Kahn} Y.~Andr\'e and B.~Kahn, {\em Nilpotence, radicaux et structures mono\"idales}, Rend. Sem. Mat. Univ.
Padova {\bf 108} (2002), 107--291, With an appendix by Peter O'Sullivan. Erratum: Rend. Sem. Mat. Univ. Padova {\bf 113} (2005), 125--128.

\bibitem{Bittner} F.~Bittner, {\em The universal Euler characteristic for varieties of characteristic zero}. Compos. Math. {\bf 140} (2004), no. 4, 1011--1032.

\bibitem{Borisov} L.~Borisov, {\em The class of the affine line is a zero divisor in the Grothendieck ring}. J. Algebr. Geom. {\bf 27}(2), 203--209 (2018)

\bibitem{Integration} A.~Chambert-Loir, J.~Nicaise and J.~Sebag, {\em Motivic integration}. Progress in Mathematics, {\bf 325}. Birkh\"auser, Springer, 2018.

\bibitem{SG} P.~Colmez and J.-P. Serre, {\em Correspondance Grothendieck-Serre}. Documents Math\'ematiques {\bf 2}, 2001.

\bibitem{Book-quadrics} R. Elman, N.~Karpenko and A.~Merkurjev, {\em The Algebraic and Geometric Theory of Quadratic Forms}. Colloquium Publications, vol. {\bf 56}. American Mathematical Society, 2008.

\bibitem{Hogadi} A.~Hogadi, {\em Products of Brauer-Severi surfaces}. Proc. AMS {\bf 137} (2009), no. 1, 45--50. 

\bibitem{Karpenko} N.~Karpenko, {\em Criteria of motivic equivalence for quadratic forms and central simple algebras}. Math. Ann. {\bf 317} (2000), no. 3, 585--611. 

\bibitem{ICM-Keller} B.~Keller, {\em On differential graded categories}. International Congress of Mathematicians (Madrid), Vol.~II,  151--190. Eur.~Math.~Soc., Z{\"u}rich (2006).

\bibitem{Kelly} G.~M.~Kelly, {\em On the radical of a category}. J. Australian Math. Soc. {\bf 4} (1964), 299--307.

\bibitem{Kollar} J.~Koll{\'a}r, {\em Conics in the Grothendieck ring}. Adv. Math. {\bf 198} (2005), no. 1, 27--35. 

\bibitem{Miami} M.~Kontsevich, {\em Mixed noncommutative motives}. Talk at the Workshop on Homological Mirror Symmetry,  Miami, 2010. Notes available at \url{www-math.mit.edu/auroux/frg/miami10-notes}.  

\bibitem{finMot} \bysame, {\em Notes on motives in finite characteristic}.  Algebra, arithmetic, and geometry: in honor of Yu. I. Manin. Vol. II,  213--247, Progr. Math., {\bf 270}, Birkh\"auser, Boston, MA, 2009. 

\bibitem{IAS} \bysame, {\em Noncommutative motives}. Talk at the IAS on the occasion of the $61^{\mathrm{st}}$ birthday of Pierre Deligne (2005). Available at \url{http://video.ias.edu/Geometry-and-Arithmetic}.    

\bibitem{Book-inv} M.-A.~Knus, A.~Merkurjev, M.~Rost, and J.-P.~Tignol, {\em The book of involutions}. With a preface in french by J. Tits. AMS Colloquium Publications, {\bf 44}, Providence, RI, 1998.

\bibitem{Lam} T.-Y.~Lam, {\em Introduction to quadratic forms over fields}. Graduate Studies in Mathematics, {\bf 67}. American Mathematical Society, Providence, RI, 2005.

\bibitem{Lunts} M.~Larsen and V.~Lunts, {\em Motivic measures and stable birational geometry}. Mosc. Math. J. {\bf 3} (2003), no.~1, 85--95, 259. 

\bibitem{LewisT} D.~Lewis and J.-P.~Tignol, {\em Classification theorems for central simple algebras with involution (with an appendix by R. Parimala)}. Manuscripta Mathematica, {\bf 100} (1999), 259--276.

\bibitem{Tignol} \bysame, {\em On the signature of an involution}. Arch. Math. {\bf 60} (1993), 128--135.

\bibitem{Manin} J.~Lieber, Y.~I.~Manin and M.~Marcolli, {\em Bost-Connes systems and $\bbF_1$-structures in Grothendieck rings, spectra, and Nori motives}. Facets of Algebraic Geometry. A collection in Honor of William Fulton's 80th birthday, vol. {\bf 2}, (2022), 147--227.

\bibitem{Sebag} Q.~Liu and J.~Sebag, {\em The Grothendieck ring of varieties and piecewise isomorphisms}. Math. Z. {\bf 265} (2010), no. 2, 321--342.

\bibitem{Manin1} Y.~Manin, {\em Correspondences, motifs and monoidal transformations}, Mat. Sb. (N.S.) {\bf 77} (119) (1968), 475--507.

\bibitem{Naumann} N.~Naumann, {\em Algebraic independence in the Grothendieck ring of varieties}. Trans. Amer. Math. Soc. {\bf 359} (2007), no.~4, 1653--1683.

\bibitem{Poonen} B.~Poonen, {\em The Grothendieck ring of varieties is not a domain}. Math. Res. Lett. {\bf 9} (2002), no.~4, 493--497.

\bibitem{Rost} M.~Rost, {\em The motive of a Pfister form}. Available at \url{https://www.math.uni-bielefeld.de/~rost/motive.html}.

\bibitem{Serre2} J.-P.~Serre {\em Local fields}. Springer Graduate Texts in Mathematics, vol. {\bf 67} (2013).

\bibitem{Serre} \bysame, {\em A course in arithmetic}. Springer Graduate Texts in Mathematics, vol. {\bf 7} (1973).

\bibitem{Springer} T. A.~Springer, {\em Sur les formes quadratiques d'indice z\'ero}. C. R. Acad. Sci. Paris {\bf 234} (1952),
1517--1519.

\bibitem{Tits} G.~Tabuada, {\em Jacques Tits motivic measure}. Math. Ann. {\bf 382} (2022), no. 3-4, 1245--1278.

\bibitem{book} \bysame, {\em Noncommutative Motives}. With a preface by Yuri I. Manin. University Lecture Series {\bf 63}. American Mathematical Society, Providence, RI, 2015.

\bibitem{Homogeneous} \bysame, {\em Additive invariants of toric and twisted projective homogeneous varieties via noncommutative motives}. Journal of Algebra, {\bf 417} (2014), 15--38.

\bibitem{Separable} G.~Tabuada and M. Van den Bergh, {\em Noncommutative motives of separable algebras}. Adv. Math. {\bf 303} (2016), 1122--1161.

\bibitem{Azumaya} \bysame, {\em Noncommutative motives of Azumaya algebras}. J. Inst. Math. Jussieu {\bf 14} (2015), no. 2, 379--403.

\bibitem{Tao} D.~Tao, {\em A variety associated to an algebra with involution}. J. Algebra {\bf 168} (1994), no.~2, 479--520.



\end{thebibliography}
\end{document}

\end{proof}